\documentclass{article}
\usepackage{graphicx}
\usepackage{amsmath}
\usepackage{amsfonts}
\usepackage{amssymb}
\newtheorem{theorem}{Theorem}
\newtheorem{corollary}{Corollary}

\newtheorem{lemma}{Lemma}
\newtheorem{proposition}{Proposition}

\newenvironment{proof}[1][Proof]{\noindent\textbf{#1.} }{\ \rule{0.5em}{0.5em}}
\begin{document}

\title{Central limit theorems for smoothed extreme value estimates\\of \ Poisson point processes boundaries}
\author{St\'ephane Girard$^{1}$ \& Ludovic Menneteau$^{2}$}
\date{ $^1$ SMS/LMC, Universit\'e Grenoble 1,\\
BP 53, 38041 Grenoble Cedex 9, France.\\
{\tt Stephane.Girard@imag.fr} \\
\vspace*{2mm}
$^2$ D\'epartement de Math\'ematiques,
Universit\'e Montpellier 2,\\
Place Eug\`ene Bataillon, 34095 Montpellier Cedex 5, France.
{\tt mennet@math.univ-montp2.fr\\}
}
\maketitle
\begin{abstract}
In this paper, we give sufficient conditions to establish central limit
theorems for boundary estimates of Poisson point processes. The considered
estimates are obtained by smoothing some bias corrected extreme values of the
point process. We show how the smoothing leads Gaussian asymptotic
distributions and therefore pointwise confidence intervals. Some
new unidimensional and multidimensional examples are provided. 
\end{abstract}

\section{Introduction}

Many proposals are given in the literature for estimating a set $S$ given a
finite random set $N$ of points drawn from the interior. This problem of
support estimation arises in classification (\textsc{Hardy} \& \textsc{Rasson}%
~(1982)), clustering problems (\textsc{Hartigan}~(1975)), discriminant
analysis (\textsc{Baufays} \& \textsc{Rasson}~(1985)), and outliers detection
(\textsc{Devroye} \& \textsc{Wise}~(1980)). Applications are also found in
image analysis. For instance, the segmentation problem can be considered under
the support estimation point of view, where the support is a connex bounded
set in $\mathbb{R}^{2}$. We also point out some applications in econometrics
(e.g. \textsc{Deprins}, \textit{et al}~(1984)). In such cases, the unknown
support can be written
\begin{equation}
S=\{(x,y):x\in E\;;\;0\leq~y~\leq~f(x)\}, \label{defS}%
\end{equation}
where $f$ is an unknown function and $E$ an arbitrary set. The set $S$ is
often called a boundary fragment, see \textsc{Korostelev} \& \textsc{Tsybakov}%
~(1993), Chapter~3. Then, the problem reduces to estimating $f$, sometimes
called the production frontier (see for instance \textsc{H\"{a}rdle}
\textit{et al}~(1995a)). The data consist of pair $(X,Y)$ where $X$ represents
the input, possibly multidimensional (labor, energy or capital), used to
produce an output $Y$ in a given firm. In such a framework, the value $f(x)$
can be interpreted as the maximum level of output which is attainable for the
level of input $x$. \textsc{Korostelev} \textit{et al}~(1995) suppose $f$ to
be increasing and concave, from economical considerations, which suggests an
adapted estimator, called the DEA (Data Envelopment Analysis) estimator. Its
asymptotic distribution is established by \textsc{Gijbels} \textit{et al}~(1999).

\noindent Here, $N$ is a Poisson point process, with observed points belonging
to a subset $S$ defined as in (\ref{defS}) where $f$ is an unknown function
which needs not to be monotone. An early paper was written by \textsc{Geffroy}%
~(1964) for independent identically distributed observations from a density.
The proposed estimator is a kind of histogram based on the extreme values of
the sample. This work was extended in two main directions.

\begin{enumerate}
\item [(a)]On the one hand, piecewise polynomials were introduced and their
optimality in an asymptotic minimax sense is proved under weak assumptions on
the rate of decrease $\beta$ of the density towards 0 and on the number $q$ of
continuous derivatives of $f$ by \textsc{Korostelev} \& \textsc{Tsybakov}%
~(1993) and by \textsc{H\"ardle} \textit{et al}~(1995b). The asymptotic
distribution is established by \textsc{Hall} \textit{et al}~(1998). Extreme
values methods are proposed by \textsc{Hall} \textit{et al}~(1997) and by
\textsc{Gijbels} \& \textsc{Peng}~(1999) to estimate the parameter $\beta$.

\item[(b)] On the other hand, different propositions for smoothing Geffroy's
estimate were made. \textsc{Girard} \& \textsc{Jacob}~(2001, 2003a, 2003b)
introduced estimates based on kernel regressions and on projection methods. In
the same spirit, \textsc{Gardes}~(2002) proposed a Faber-Shauder estimate. In
each case, the consistency and the limit distribution of the estimator are established.
\end{enumerate}

\noindent\ Finally, the work of \textsc{Mammen \& Tsybakov}~(1995) offers a
general framework for comparing the estimates of type (a) or (b). The optimal
rates of convergence are derived for estimates of boundaries which have a
smooth parametrisation. The existence of estimates reaching these optimal
rates of convergence is proved by the minimization of contrast criteria over
classes of functions.

\noindent Here, we introduce new estimates of type (b). The considered
estimates are obtained by smoothing the bias corrected extreme values of the
Poisson process (see {\sc Menneteau}~(2003a) for related work in the iid
setting). 

This approach offers several advantages. First, the bias
correction allows to overcome the classical limitation due to the fact that
the data lie below the boundary. Second, the smoothing permits to obtain
Gaussian asymptotic distributions. Therefore, it is straightforward to obtain
pointwise confidence intervals for $f(x)$ all the more so as our estimates
benefit from explicit forms and are easy to compute. Finally, let us note that
our estimates offer new features compared to those quoted in (b): i) They are
not dedicated to unidimensional boundary estimation problems since there is
no restriction on the set $E$ in (\ref{defS}),
ii) the bias correction is
different and thus, iii) the intensity measure of the point process can be more general,
iv) the smoothing is achieved with more general weight functions allowing
v) better speeds of convergence than the previous estimates quoted in (b).

\section{The boundary estimate}

Let $(E,\mathcal{E},\nu)$ be a probability space and $f:(E,\mathcal{E}%
)\rightarrow\left(  \mathbb{R}^+,\mathcal{B}\left(  \mathbb{R}^+\right)  \right)
$ a measurable function, where $\mathcal{B}\left(  \mathbb{R}\right)  $ is the
Borel $\sigma$-algebra on $\mathbb{R}$. Consider $S=\{(x,y)\in E\times
\mathbb{R},0\leq y\leq f(x)\}$ and the sequence of Poisson point processes
\[
N_{n}=\left\{  N_{n}\left(  D\right)  :D\in\mathcal{E}\otimes\mathcal{B}%
\left(  \mathbb{R}^+\right)  \right\}  ,\;\;n\geq1,
\]
with intensity measure
\begin{equation}
nc(\nu\otimes\lambda)\,\mathbf{1}_{S}, \label{int}%
\end{equation}
where $c>0$, and $\lambda$ is the Lebesgue measure on $\mathbb{R}^{+}$. Let
$\{(X_{n,i},Y_{n,i}),\;1\leq i\leq N_{n}(S)\}$ be the set of points associated
to the point process. Our aim is then to estimate $S$ via an estimation of
$f$. Let $k_{n}\uparrow\infty$ and denote by $\{I_{n,r}:\;1\leq r\leq k_{n}\}$
a measurable partition of $E$. For all $1\leq r\leq k_{n}$, note $\nu
_{n,r}=\nu(I_{n,r})$,
\[
D_{n,r}=\{(x,y):x\in I_{n,r},\;0\leq y\leq f(x)\}
\]
the cell of $S$ built on $I_{n,r}$ and $N_{n,r}=N_{n}(D_{n,r})$. We introduce
the extreme values
\[
Y_{n,r}^{\ast}=\max\{Y_{n,i}:(X_{n,i},Y_{n,i})\in D_{n,r}\},
\]
if $N_{n,r}\neq0$ and $Y_{n,r}^{\ast}=0$ otherwise. In the following, the
convention $0\times\infty=0$ is adopted. For $x\in E,$ our estimator of
$f\left(  x\right)  $ is
\begin{equation}
\widehat{f}_{n}\left(  x\right)  =\sum_{r=1}^{k_{n}}\nu_{n,r}\kappa
_{n,r}(x)\left(  1+\frac{1}{N_{n,r}}\right)  Y_{n,r}^{\ast}, \label{defest}%
\end{equation}
where $\kappa_{n,r}:E\rightarrow\mathbb{R}$ is a weighting function
determining the nature of the smoothing introduced in the estimate. In the
next section, some general conditions are imposed on $\kappa_{n,r}$ and
examples are provided in Section~\ref{applications}. It is well-known that
$Y_{n,r}^{\ast}$ is an estimator of the maximum of $f$ on $I_{n,r}$ with
negative bias. The use of the random variable $(1+N_{n,r}^{-1})Y_{n,r}^{\ast}$
allows to reduce this bias. This bias correction is motivated by the remark
that, conditionally on $N_{n,r}$, $Y_{n,r}^{\ast}$ has approximatively the
same distribution as the maximum of $N_{n,r}$ independent random variables
uniformly distributed on $[0,\min\{f(x):x\in I_{n,r}\}]$ (see Lemma \ref{mom}
ii) below). Therefore, $\widehat{f}_{n}$ appears as a linear combination of
extreme value estimates of sampled values of $f$. The asymptotic properties of
$\widehat{f}_{n}$ are established in Section~\ref{secmain}, and proved in
Section~\ref{proofs}. Illustrations are presented in
Section~\ref{applications} with general kernel estimates including
Parzen-Rosenblatt and Dirichlet kernels.  

\section{Main results \label{secmain}}

Define $\nu_{n}=\min\{\nu_{n,r},\;1\leq r\leq k_{n}\}$ and
\[
\kappa_{n}(x)=\left(  \sum_{r=1}^{k_{n}}\kappa_{n,r}^{2}(x)\right)
^{1/2},\;x\in E.
\]
Let $m=\sup\{\alpha>0:\nu\left(  \left\{  f<\alpha\right\}  \right)  =0\}$ and
$M=\inf\{\alpha>0:\nu\left(  \left\{  f>\alpha\right\}  \right)  =0\}$ be the
$\nu$-essential infimum and supremum of $f$ on $E.$ Similarly, for all $1\leq
r\leq k_{n},$ let
\[
m_{n,r}=\sup\{\alpha>0:\nu\left(  \left\{  f<\alpha\right\}  \cap
I_{n,r}\right)  =0\},
\]%
\[
M_{n,r}=\inf\{\alpha>0:\nu\left(  \left\{  f>\alpha\right\}  \cap
I_{n,r}\right)  =0\}
\]
and
\[
f_{n,r}=\nu_{n,r}^{-1}\;\int_{I_{n,r}}f\;d\nu
\]
be respectively the $\nu$-essential infimum, the $\nu$-essential supremum and
the mean value of $f$ on $I_{n,r}$ and define the
$\nu$-essential oscillation of $f$ on $I_{n,r}$ by
\[
\Delta_{n}=\max\{M_{n,r}-m_{n,r},\;1\leq r\leq k_{n}\}.
\]
Let us highlight that, in most applications (see Section~\ref{applications}),
$E$ is a subset of $\mathbb{R}^{d}$, $\nu$ is absolutely continuous with
respect to Lebesgue measure and $f$ is continuous. Hence, all essential
infima and essential suprema considered below reduce to the classical minima
and maxima. \\ Finally, set $w_{n,r}(x)=\kappa_{n,r}\left(  x\right)
/\kappa_{n}(x).$ We consider the following series of assumptions:\newline
$\left(  \mathrm{H.1}\right)  \;k_{n}\uparrow\infty$ and $n\nu_{n}%
\rightarrow\infty$ as $n\rightarrow\infty$.\medskip\newline $\left(
\mathrm{H.2}\right)  \;$~$0<m\leq M<+\infty$ and
\[
\delta_{n}:=\max_{1\leq r\leq k_{n}}\nu_{n,r}(M_{n,r}-m_{n,r})=o\left(
{1/n}\right)  \;\mathrm{as\;}n\rightarrow\infty.
\]
There exists $F\subset E$ such that: \medskip\newline $\left(  \mathrm{H.3}%
\right)  \;$For each $(x_{1},...,x_{p})\subset F$, there exists a covariance
matrix $\Sigma_{(x_{1},\dots,x_{p})}=[\sigma(x_{i},x_{j})]_{1\leq i,j\leq p}$
in ${}\mathbb{R}^{p}$ such that for all $1\leq i$, $j\leq p$,
\[
\sum_{r=1}^{k_{n}}w_{n,r}(x_{i})w_{n,r}(x_{j})\rightarrow\;\sigma(x_{i}%
,x_{j})\;\mathrm{as\;}n\rightarrow\infty.
\]
$\left(  \mathrm{H.4}\right)  \;$For all $x\in F$,
\[
\max_{1\leq r\leq k_{n}}|w_{n,r}(x)|\rightarrow0\;\mathrm{as\;}n\rightarrow
\infty.
\]
$\left(  \mathrm{H.5}\right)  \;$For all $x\in F$,
\[
\left|  \sum_{r=1}^{k_{n}}\nu_{n,r}\kappa_{n,r}(x)f_{n,r}-f(x)\right|
=o\left(  {\frac{\kappa_{n}(x)}{n}}\right)  \;\mathrm{as\;}n\rightarrow
\infty.
\]
$\left(  \mathrm{H.6}\right)  \;$For all $x\in F$,
\[
\sum_{r=1}^{k_{n}}|w_{n,r}(x)|\max\left(  \left(  n\delta_{n}\right)
^{2},n\nu_{n}\exp\left(  -mcn\nu_{n}\right)  ,\Delta_{n}\right)
\rightarrow0\;\mathrm{as\;}n\rightarrow\infty.
\]

\noindent Before proceeding, let us comment on the assumptions. $\left(
\mathrm{H.1}\right)  $--$\left(  \mathrm{H.4}\right)  $~are devoted to the
control of the centered estimator $\widehat{f}_{n}\left(  x\right)
-\mathbb{E}(\widehat{f}_{n}\left(  x\right)  )$. Assumption $\left(
\mathrm{H.1}\right)  $~imposes that the mean number of points in each
$D_{n,r}$ goes to infinity. $\left(  \mathrm{H.2}\right)  $\ requires the
unknown function $f$ to be bounded away from 0. It also imposes that the mean
number of points in $D_{n,r}$ above $m_{n,r}$ converges to 0. Note that
$\left(  \mathrm{H.1}\right)  $~and $\left(  \mathrm{H.2}\right)  $~force the
$\nu$-essential oscillation of $f$ on $I_{n,r}$ to converge uniformly to~0:
$\Delta_{n}\rightarrow0$ as $n\rightarrow\infty$. $\left(  \mathrm{H.3}%
\right)  $ is devoted to the multivariate aspects of the limit theorems.
$\left(  \mathrm{H.4}\right)  $ imposes to the weight functions $\kappa
_{n,r}(x)$ in the linear combination (\ref{defest}) to be approximatively of
the same order. This is a natural condition to obtain an asymptotic Gaussian
behavior. These assumptions are easy to verify in practice since they involve
either $f(x)$ or $\kappa_{n,r}(x)$ without mixing these two quantities.
Assumptions $\left(  \mathrm{H.5}\right)  $~and $\left(  \mathrm{H.6}\right)
$~are devoted to the control of the bias term $\mathbb{E}(\widehat{f}%
_{n}\left(  x\right)  )-f(x)$. They prevent it to be too important with
respect to the variance of the estimate (which will reveal to be of order
$\kappa_{n}(x)/n$). Consequently, these two assumptions involve both the
unknown function $f(x)$ and the weight functions $\kappa_{n,r}(x)$. Finally,
$\left(  \mathrm{H.6}\right)  $~can be looked at as a stronger version of
$\left(  \mathrm{H.2}\right)  $.

\noindent Our first result states the multivariate central limit theorem for
$\widehat{f}_{n}(x).$

\begin{theorem}
\label{TCL}Under assumptions $\left(  \mathrm{H.1}\right)  $-$\left(
\mathrm{H.6}\right),$ and for all $\left(  x_{1},...,x_{p}\right)  \subset
F,$%
\[
\left\{  \frac{nc}{\kappa_{n}\left(  x_{j}\right)  }\left(  \widehat{f}%
_{n}\left(  x_{j}\right)  -f\left(  x_{j}\right)  \right)  :1\leq j\leq
p\right\}  \underset{\mathcal{D}}{\rightarrow}N\left(  0,\Sigma_{\left(
x_{1},...,x_{p}\right)  }\right)  ,
\]
where $c$ is defined in $\left(  \ref{int}\right)  ,$\ $\underset{\mathcal{D}%
}{\rightarrow}$ denotes the convergence in distribution and $N\left(
0,\Sigma_{\left(  x_{1},...,x_{p}\right)  }\right)  $ is the centered Gaussian
distribution in $\mathbb{R}^{p},$ with covariance matrix $\Sigma_{\left(
x_{1},...,x_{p}\right)  }.$
\end{theorem}

\noindent In practice, $c$ is not known and has to be estimated. In this aim,
we introduce $\widehat{c}_{n}=N_{n}(S)/(n\widehat{a}_{n})$, where
\[
\widehat{a}_{n}=\ \sum_{r=1}^{k_{n}}\nu_{n,r}\left(  1+\frac{1}{N_{n,r}%
}\right)  Y_{n,r}^{\ast}%
\]
is an estimator of $a=\int_{E}fd\nu$. We then have the following corollary:

\begin{corollary}
\label{coromain}Theorem \ref{TCL} holds when $c$ is replaced by $\widehat
{c}_{n}$.
\end{corollary}

\noindent For all $x\in F,$ this leads to an explicit asymptotic $\gamma$th
confidence interval for $f(x)$:
\[
\left[  \sum_{r=1}^{k_{n}}\nu_{n,r}\left(  \kappa_{n,r}(x)-z_{\gamma}%
\frac{\kappa_{n}(x)}{N_{n}(S)}\right)  \left(  1+\frac{1}{N_{n,r}}\right)
Y_{n,r}^{\ast},\sum_{r=1}^{k_{n}}\nu_{n,r}\left(  \kappa_{n,r}(x)+z_{\gamma
}\frac{\kappa_{n}(x)}{N_{n}(S)}\right)  \left(  1+\frac{1}{N_{n,r}}\right)
Y_{n,r}^{\ast}\right]  ,
\]
where $z_{\gamma}$ is the $(\gamma+1)/2$th quantile of the $N(0,1)$
distribution. Note that the computation of this interval does not require a
bootstrap procedure as for instance in \textsc{Hall} \textit{et al}~(1998).

\noindent\textbf{Remark.} In the case where the measure $\nu$ is unknown, it
is natural to introduce the boundary estimate:
\begin{equation}
\overset{\circ}{f}_{n}(x)=\sum_{r=1}^{k_{n}} \overset{\circ}{\nu}_{n,r}\kappa
_{n,r}(x)\left(  1+\frac{1}{N_{n,r}}\right)  Y_{n,r}^{\ast}, \label{ftilde}%
\end{equation}
where $\overset{\circ}{\nu}_{n,r}$ is an estimator of $\nu_{n,r}$. If no prior
information is available on $\nu$, one can use the following non-parametric
estimate:
\[
\overset{\circ}{\nu}_{n,r}=\frac{N_{n,r}}{ncY_{n,r}^{\ast}(1+N_{n,r}^{-1})},
\]
leading to
\[
\overset{\circ}{f}_{n}(x)=\sum_{r=1}^{k_{n}}
\kappa _{n,r}(x) \frac{N_{n,r}}{nc},
\]
which has been first introduced by \textsc{Jacob} \& \textsc{Suquet}~(1995)
with a particular choice of the weighting function $\kappa_{n,r}$ and when
$\nu$ is the Lebesgue measure. Let us note that Theorem~\ref{TCL} does not
hold for $\overset{\circ}{f}_{n}$ since it converges slower than
${\hat f}_{n}$, see \textsc{Jacob} \& \textsc{Suquet}~(1995),
Theorem~7. If $\nu$ is assumed to belong to a parametric family, another
versions of (\ref{ftilde}) can be used, leading to semi-parametric estimates
of $f$.

\section{Proofs of the main results \label{proofs}}

The proofs are built as follows. First, we establish a multivariate central
limit theorem for the finite dimensional projection of the centered process
\[
\frac{nc}{\kappa_{n}\left(  x\right)  }\left(  \widehat{f}_{n}\left(
x\right)  -\mathbb{E}\left(  \widehat{f}_{n}\left(  x\right)  \right)
\right)  ,\;\;x\in F
\]
(see Proposition~\ref{TCLseul} below). To this aim, by the general framework
of the appendix (Theorem~\ref{GenTCL}) it is sufficient to control the
centered moments of
\[
\xi_{n,r}=\left(  1+\frac{1}{N_{n,r}}\right)  Z_{n,r}^{\ast}=\left(
1+\frac{1}{N_{n,r}}\right)  nc\nu_{n,r}Y_{n,r}^{\ast}.
\]
This is achieved in Lemma \ref{moml}. In a second time, we establish that the
bias term
\[
\frac{nc}{\kappa_{n}\left(  x\right)  }\left(  \mathbb{E}\left(  \widehat
{f}_{n}\left(  x\right)  \right)  -f\left(  x\right)  \right)
\]
vanishes when $n\uparrow\infty$ (see Proposition \ref{centr}). Finally, we
prove in Lemma \ref{remplace1} that $c$ can be replaced by $\widehat{c}_{n}$
in the multivariate central limit theorem. Before that, we introduce some new
notations and definitions needed for our proofs. For all $1\leq r\leq k_{n}$,
each cell $D_{n,r}$ can be splitted as $D_{n,r}=\widetilde{D}_{n,r}\cup
D_{n,r}^{-} \cup D_{n,r}^{+}$ where

\begin{itemize}
\item $\widetilde{D}_{n,r}=\left\{  \left(  x,y\right)  \in I_{n,r}%
\times[0,m_{n,r}],\;f\left(  x\right)  <m_{n,r}\right\}  $,

\item $D_{n,r}^{-}=\left\{  \left(  x,y\right)  \in D_{n,r},\;0\leq y\leq
m_{n,r}\right\}  =\left(  I_{n,r}\times\left[  0,m_{n,r}\right]  \right)
\diagdown\widetilde{D}_{n,r},$

\item $D_{n,r}^{+}=\left\{  \left(  x,y\right)  :x\in I_{n,r},\;m_{n,r}<y\leq
f\left(  x\right)  \right\}  .$
\end{itemize}
Moreover, for all $1\leq r\leq k_{n}$, set
\begin{itemize}
\item $\lambda_{n,r}=nc\nu_{n,r}m_{n,r}$,

\item $\mu_{n,r}=nc\nu_{n,r}f_{n,r}$,

\item $N_{n,r}^{-}=N_{n}\left(  D_{n,r}^{-}\right)  $, $N_{n,r}^{+}%
=N_{n}\left(  D_{n,r}^{+}\right)  $,

\item $Z_{n,r}^{-}=nc\nu_{n,r}\max\{Y_{n,i}:(X_{n,i},Y_{n,i})\in D_{n,r}%
^{-}\}$, if $N_{n,r}^{-}\neq0$, and $Z_{n,r}^{-}=0$ otherwise,\newline
$Z_{n,r}^{+}=nc\nu_{n,r}\max\{Y_{n,i}:(X_{n,i},Y_{n,i})\in D_{n,r}^{+}\}$, if
$N_{n,r}^{+}\neq0$, and $Z_{n,r}^{+}=0$ otherwise,

\item $\xi_{n,r}^{-}=\left(  1+\frac{1}{N_{n,r}^{-}}\right)  Z_{n,r}^{-}.$
\end{itemize}

\noindent Some technical results are collected in Lemma~\ref{mom}. The second
of them is the key tool for proving the following ones. It states that,
conditionally on $N_{n,r}^{-}$, $Z_{n,r}^{-}$ has the same distribution as the
maximum of $N_{n,r}^{-}$ independent random variables uniformly distributed on
$[0,\lambda_{n,r}]$. This motivates the bias correction in (\ref{defest}).

\begin{lemma}
\label{mom}Under assumptions $\left(  \mathrm{H.1}\right)  $\ and\ $\left(
\mathrm{H.2}\right)  $\ we have\newline i) $\underset{1\leq r\leq k_{n}}{\max
}\mathbb{P}\left(  N_{n,r}^{+}>0\right)  =O\left(  n\delta_{n}\right)
=o\left(  1\right)  .$ \newline ii) For all $1\leq r\leq k_{n},$ and any
$t\in\left[  0,\lambda_{n,r}\right]  $, $\mathbb{P}\left(  Z_{n,r}^{-}\leq
t\mid N_{n,r}^{-}\right)  =\left(  \frac{t}{\lambda_{n,r}}\right)
^{N_{n,r}^{-}}$. \newline iii) For all $1\leq r\leq k_{n},$ and any
$t\in\left[  0,\lambda_{n,r}\right]  $, $\mathbb{P}\left(  Z_{n,r}^{-}\leq
t\right)  =\exp(t-\lambda_{n,r}).$\newline iv) For all $1\leq r\leq k_{n},$
$\mathbb{E}\left(  Z_{n,r}^{-}\right)  =\lambda_{n,r}-\left(  1-e^{-\lambda
_{n,r}}\right)  .$\newline v) For all $1\leq r\leq k_{n},$ $\mathbb{V}\left(
Z_{n,r}^{-}\right)  =1-2\lambda_{n,r}e^{-\lambda_{n,r}}-e^{-2\lambda_{n,r}}%
.$\newline vi) For all $1\leq r\leq k_{n},$ $\mathbb{E}\left(  \frac
{Z_{n,r}^{-}}{N_{n,r}^{-}}\right)  =1-e^{-\lambda_{n,r}}\left(  1+\lambda
_{n,r}\right)  .$\newline vii) For all $\ell\geq1$ and $1\leq r\leq k_{n},$
$\mathbb{E}\left(  (\frac{Z_{n,r}^{-}}{N_{n,r}^{-}})^{\ell}\right)  \leq
\ell!.$\newline viii) $\underset{1\leq r\leq k_{n}}{\max}\mathbb{V}\left(
\frac{Z_{n,r}^{-}}{N_{n,r}^{-}}\right)  =o(1).$\newline ix) For all $\ell
\geq1$ and $1\leq r\leq k_{n},$ $\mathbb{E}\left(  \left|  Z_{n,r}%
^{-}-\mathbb{E}\left(  Z_{n,r}^{-}\right)  \right|  ^{\ell}\right)  \leq
1+\ell!.$\newline x) For all $\ell\geq1$, $\underset{1\leq r\leq k_{n}}{\max
}\mathbb{E}\left(  |\xi_{n,r}-\xi_{n,r}^{-}|^{\ell}\right)  =O(n\delta_{n}%
).$\newline xi) $\underset{1\leq r\leq k_{n}}{\max}\left|  \mathbb{E}\left(
Z_{n,r}^{\ast}\right)  -\mu_{n,r}+1\right|  =O\left(  \max\left(  \exp\left(
-mnc\nu_{n}\right)  ;\left(  n\delta_{n}\right)  ^{2}\right)  \right)  .$
\end{lemma}

\begin{proof}
i) is straightforward.\newline ii) First, note that, since $(\nu\otimes
\lambda)(\widetilde{D}_{n,r})=0,$ then for all $A\in\mathcal{B}\left(
\mathbb{R}^{+}\right)  ,$
\[
N_{n}\left(  D_{n,r}^{-}\cap\left(  I_{n,r}\times A\right)  \right)
=N_{n}\left(  I_{n,r}\times\left(  \left[  0,m_{n,r}\right]  \cap A\right)
\right)  \;\mathrm{a.s.}%
\]
which is a Poisson random variable with mean $nc\nu_{n,r}\lambda\left(
\left[  0,m_{n,r}\right]  \cap A\right)  $ . Second, set $t\in\left[
0,\lambda_{n,r}\right]  ,$ and define $t_{n,r}=t/(nc\nu_{n,r})$. Then, for all
$q\geq1,$%
\begin{align}
&  \mathbb{P}\left(  Z_{n,r}^{-}\leq t\mid N_{n,r}^{-}=q\right) \nonumber\\
=  &  \frac{\mathbb{P}\left(  N_{n}\left(  D_{n,r}^{-}\cap\left(
I_{n,r}\times\left[  0,t_{n,r}\right]  \right)  \right)  =q\right)
\;\mathbb{P}\left(  N_{n}\left(  D_{n,r}^{-}\cap\left(  I_{n,r}\times\left(
t_{n,r},+\infty\right)  \right)  \right)  =0\right)  }{\mathbb{P}\left(
N_{n}\left(  D_{n,r}^{-}\right)  =q\right)  }\nonumber\\
=  &  \left(  \frac{t}{\lambda_{n,r}}\right)  ^{q}. \label{condq}%
\end{align}
Noticing that $\left(  \ref{condq}\right)  $ is obvious when $q=0$ gives the
result.\newline iii)--v) are deduced from ii) by easy calculations.
\newline vi) It follows from ii) that
\begin{align*}
\mathbb{E}\left(  \frac{Z_{n,r}^{-}}{N_{n,r}^{-}}\right)   &  =\mathbb{E}%
\left(  \frac{\mathbf{1}_{\left\{  N_{n,r}^{-}>0\right\}  }}{N_{n,r}^{-}%
}\mathbb{E}\left(  Z_{n,r}^{-}\mid N_{n,r}^{-}\right)  \right)  =\lambda
_{n,r}\mathbb{E}\left(  \frac{\mathbf{1}_{\left\{  N_{n,r}^{-}>0\right\}  }%
}{N_{n,r}^{-}+1}\right)  =\underset{q=1}{\overset{\infty}{\sum}}\frac
{\lambda_{n,r}^{q+1}}{\left(  q+1\right)  !}e^{-\lambda_{n,r}}\\
&  =1-e^{-\lambda_{n,r}}\left(  1+\lambda_{n,r}\right)  .
\end{align*}
vii) We have,
\begin{align}
\mathbb{E}\left(  \left(  \frac{Z_{n,r}^{-}}{N_{n,r}^{-}}\right)  ^{\ell
}\right)   &  =\mathbb{E}\left(  \left(  \frac{1}{N_{n,r}^{-}}\right)  ^{\ell
}\mathbb{E}\left(  \left(  Z_{n,r}^{-}\right)  ^{\ell}\mid N_{n,r}^{-}\right)
\right)  =\lambda_{n,r}^{\ell}\mathbb{E}\left(  \frac{\mathbf{1}_{\left\{
N_{n,r}^{-}>0\right\}  }}{\left(  N_{n,r}^{-}\right)  ^{\ell-1}\left(
\ell+N_{n,r}^{-}\right)  }\right) \nonumber\\
&  =\ell!\underset{q=1}{\overset{\infty}{\sum}}\frac{\left(  q+\ell-1\right)
!}{\ell!q^{\ell-1}q!}\frac{\lambda_{n,r}^{q+\ell}}{\left(  q+\ell\right)
!}e^{-\lambda_{n,r}}\label{g3esp2}\\
&  \leq\ell!,\nonumber
\end{align}
since for all $\ell$ and $q$, $(q+\ell-1)!\leq\ell!q^{\ell-1}q!.$%
\newline viii) By $\left(  \ref{g3esp2}\right)  $ with $\ell=2,$
\begin{align*}
\mathbb{E}\left(  \left(  \frac{Z_{n,r}^{-}}{N_{n,r}^{-}}\right)
^{2}\right)   &  =\underset{q=1}{\overset{\infty}{\sum}}\left(  1+\frac{1}%
{q}\right)  \frac{\lambda_{n,r}^{q+2}}{\left(  q+2\right)  !}e^{-\lambda
_{n,r}}\\
&  =1-e^{-\lambda_{n,r}}\left(  1+\lambda_{n,r}+\frac{\lambda_{n,r}^{2}}%
{2}\right)  +\underset{q=1}{\overset{\infty}{\sum}}\frac{1}{q}\frac
{\lambda_{n,r}^{q+2}}{\left(  q+2\right)  !}e^{-\lambda_{n,r}}.
\end{align*}
Now, since,
\[
\underset{1\leq r\leq k_{n}}{\max}e^{-\lambda_{n,r}}\left(  1+\lambda
_{n,r}+\frac{\lambda_{n,r}^{2}}{2}\right)  \leq\underset{\lambda\geq
mnc\nu_{n}}{\max}e^{-\lambda}\left(  1+\lambda+\frac{\lambda^{2}}{2}\right)
=o\left(  1\right)  ,
\]
and for all $Q\geq1,$
\begin{align*}
\underset{1\leq r\leq k_{n}}{\max}\underset{q=1}{\overset{\infty}{\sum}}%
\frac{1}{q}\frac{\lambda_{n,r}^{q+2}}{\left(  q+2\right)  !}e^{-\lambda
_{n,r}}  &  \leq\underset{1\leq r\leq k_{n}}{\max}\underset{q=1}{\overset
{Q}{\sum}}\frac{\lambda_{n,r}^{q+2}}{\left(  q+2\right)  !}e^{-\lambda_{n,r}%
}+\frac{1}{Q}\underset{1\leq r\leq k_{n}}{\max}\underset{q=1}{\overset{\infty
}{\sum}}\frac{\lambda_{n,r}^{q+2}}{\left(  q+2\right)  !}e^{-\lambda_{n,r}}\\
&  \leq o\left(  1\right)  +\frac{1}{Q},
\end{align*}
it follows that
\begin{equation}
\underset{1\leq r\leq k_{n}}{\max}\left|  \mathbb{E}\left(  \left(
\frac{Z_{n,r}^{-}}{N_{n,r}^{-}}\right)  ^{2}\right)  -1\right|  =o(1).
\label{esp2}%
\end{equation}
Collecting (\ref{esp2}) and vi) gives the result.\newline ix) Note that
\[
\mathbb{E}\left(  \left|  Z_{n,r}^{-}-\mathbb{E}\left(  Z_{n,r}^{-}\right)
\right|  ^{\ell}\right)  =\ell\int_{0}^{+\infty}t^{\ell-1}\mathbb{P}\left(
\left|  Z_{n,r}^{-}-\mathbb{E}\left(  Z_{n,r}^{-}\right)  \right|  >t\right)
dt.
\]
Moreover, by iii) and iv),
\begin{align*}
\mathbb{P}\left(  \left|  Z_{n,r}^{-}-\mathbb{E}\left(  Z_{n,r}^{-}\right)
\right|  >t\right)   &  =\mathbb{P}\left(  Z_{n,r}^{-}>\mathbb{E}\left(
Z_{n,r}^{-}\right)  +t\right)  +\mathbb{P}\left(  Z_{n,r}^{-}<\mathbb{E}%
\left(  Z_{n,r}^{-}\right)  -t\right) \\
&  =\left[  1-\exp\left(  t-\lambda_{n,r}+\mathbb{E}\left(  Z_{n,r}%
^{-}\right)  \right)  \right]  \mathbf{1}_{\left[  0,\lambda_{n,r}%
-\mathbb{E}\left(  Z_{n,r}^{-}\right)  \right]  }(t)\\
&  +\exp\left(  -t-\lambda_{n,r}+\mathbb{E}\left(  Z_{n,r}^{-}\right)
\right)  \mathbf{1}_{\left[  0,\mathbb{E}\left(  Z_{n,r}^{-}\right)  \right]
}(t)\\
&  \leq\mathbf{1}_{\left[  0,1\right]  }\left(  t\right)  +\exp\left(
e^{-\lambda_{n,r}}-1\right)  e^{-t}.
\end{align*}
Hence,
\begin{align*}
\mathbb{E}\left(  \left|  Z_{n,r}^{-}-\mathbb{E}\left(  Z_{n,r}^{-}\right)
\right|  ^{\ell}\right)   &  \leq\ell\int_{0}^{1}t^{\ell-1}dt+\exp\left(
e^{-\lambda_{n,r}}-1\right)  \ell\int_{0}^{+\infty}t^{\ell-1}e^{-t}dt\\
&  \leq1+\exp\left(  e^{-\lambda_{n,r}}-1\right)  \ell!.
\end{align*}
x) Since $Z_{n,r}^{\ast}=\left(  Z_{n,r}^{+}-Z_{n,r}^{-}\right)
\mathbf{1}_{\left\{  N_{n,r}^{+}>0\right\}  }+Z_{n,r}^{-}$, we get
\begin{align*}
\xi_{n,r}-\xi_{n,r}^{-}  &  =\left(  1+\frac{1}{N_{n,r}}\right)  \left(
Z_{n,r}^{+}-Z_{n,r}^{-}\right)  \mathbf{1}_{\left\{  N_{n,r}^{+}>0\right\}
}+Z_{n,r}^{-}\left(  \frac{1}{N_{n,r}}-\frac{1}{N_{n,r}^{-}}\right) \\
&  =\left(  1+\frac{1}{N_{n,r}}\right)  \left(  Z_{n,r}^{+}-\lambda
_{n,r}\right)  \mathbf{1}_{\left\{  N_{n,r}^{+}>0\right\}  }+\left(
1+\frac{1}{N_{n,r}}\right)  \left(  \lambda_{n,r}-\mathbb{E}\left(
Z_{n,r}^{-}\right)  \right)  \mathbf{1}_{\left\{  N_{n,r}^{+}>0\right\}  }\\
&  +\left(  1+\frac{1}{N_{n,r}}\right)  \left(  \mathbb{E}\left(  Z_{n,r}%
^{-}\right)  -Z_{n,r}^{-}\right)  \mathbf{1}_{\left\{  N_{n,r}^{+}>0\right\}
}+Z_{n,r}^{-}\left(  \frac{1}{N_{n,r}}-\frac{1}{N_{n,r}^{-}}\right) \\
 & :=    \gamma_{n,r,1}+\gamma_{n,r,2}+\gamma_{n,r,3}+\gamma_{n,r,4}.
\end{align*}
Hence, by the triangle inequality,
\begin{equation}
\underset{1\leq r\leq k_{n}}{\max}\;\mathbb{E}\left(  \left|  \xi_{n,r}%
-\xi_{n,r}^{-}\right|  ^{\ell}\right)  ^{1/\ell}\leq\sum_{j=1}^{4}%
\underset{1\leq r\leq k_{n}}{\max}\mathbb{E}\left(  \left|  \gamma
_{n,r,j}\right|  ^{\ell}\right)  ^{1/\ell}. \label{dobe1}%
\end{equation}
First,
\begin{equation}
\underset{1\leq r\leq k_{n}}{\max}\mathbb{E}\left(  \left|  \gamma
_{n,r,1}\right|  ^{\ell}\right)  ^{1/\ell}\leq2nc\underset{1\leq r\leq k_{n}%
}{\max}\nu_{n,r}\left(  M_{n,r}-m_{n,r}\right)  \mathbb{P}\left(  N_{n,r}%
^{+}>0\right)  =o\left(  n\delta_{n}\right)  . \label{dobe2}%
\end{equation}
Second, by iv):
\begin{equation}
\underset{1\leq r\leq k_{n}}{\max}\mathbb{E}\left(  \left|  \gamma
_{n,r,2}\right|  ^{\ell}\right)  ^{1/\ell}\leq2\underset{1\leq r\leq k_{n}%
}{\max}\mathbb{P}\left(  N_{n,r}^{+}>0\right)  =O\left(  n\delta_{n}\right)  .
\label{dobe3}%
\end{equation}
Third, the independence of $N_{n,r}^{+}$ and $Z_{n,r}^{-}$ yields with ix):
\begin{equation}
\underset{1\leq r\leq k_{n}}{\max}\mathbb{E}\left(  \left|  \gamma
_{n,r,3}\right|  ^{\ell}\right)  ^{1/\ell}\leq2\left(  1+\ell!\right)
^{1/\ell}\;\underset{1\leq r\leq k_{n}}{\max}\mathbb{P}\left(  N_{n,r}%
^{+}>0\right)  =O\left(  n\delta_{n}\right)  . \label{dobe4}%
\end{equation}
Finally, since $\left|  \gamma_{n,r,4}\right|  \leq\lambda_{n,r}(N_{n,r}%
^{-}+1)^{-1}\mathbf{1}_{\left\{  N_{n,r}^{+}>0\right\}  }$ and taking into
account that
\begin{equation}
\mathbb{E}\left(  \left(  N_{n,r}^{-}+1\right)  ^{-\ell}\right)  =\ell
!\lambda_{n,r}^{-\ell}\underset{q=0}{\overset{\infty}{\sum}}\frac{\left(
q+\ell\right)  !}{\left(  q+1\right)  ^{\ell}\ell!q!}\frac{\lambda
_{n,r}^{q+\ell}}{\left(  q+\ell\right)  !}e^{-\lambda_{n,r}}\leq\ell
!\lambda_{n,r}^{-\ell}, \label{dobe6}%
\end{equation}
(where we used the fact that for all $\ell$ and $q$, $\left(  q+\ell\right)
!\leq\left(  q+1\right)  ^{\ell}\ell!q!$), we get
\begin{align}
\underset{1\leq r\leq k_{n}}{\max}\mathbb{E}\left(  \left|  \gamma
_{n,r,4}\right|  ^{\ell}\right)  ^{1/\ell}  &  \leq\;\underset{1\leq r\leq
k_{n}}{\max}\lambda_{n,r}\mathbb{E}\left(  \left(  N_{n,r}^{-}+1\right)
^{-\ell}\right)  ^{1/\ell}\mathbb{P}\left(  N_{n,r}^{+}>0\right) \nonumber\\
&  \leq\;\left(  \ell!\right)  ^{1/\ell}\;\underset{1\leq r\leq k_{n}}{\max
}\mathbb{P}\left(  N_{n,r}^{+}>0\right)  =O\left(  n\delta_{n}\right)  .
\label{dobe5}%
\end{align}
The result is a consequence of  $\left(  \ref{dobe1}\right)  -\left(
\ref{dobe5}\right)  .$\newline xi) Note that
\begin{align*}
\mathbb{E}\left(  Y_{n,r}^{\ast}\right)   &  =\int_{0}^{M_{n,r}}%
\mathbb{P}\left(  Y_{n,r}^{\ast}\geq u\right)  du\\
&  =\int_{0}^{m_{n,r}}\left(  1-\mathbb{P}\left(  N_{n}\left(  D_{n,r}%
\cap\left(  I_{n,r}\times\left[  u,M_{n,r}\right]  \right)  \right)  \right)
\right)  du+\int_{m_{n,r}}^{M_{n,r}}\mathbb{P}\left(  Y_{n,r}^{\ast}>u\right)
du\\
&  =\int_{0}^{m_{n,r}}\left(  1-\exp\left(  -nc\nu_{n,r}f_{n,r}+nc\nu
_{n,r}u\right)  \right)  du+\int_{m_{n,r}}^{M_{n,r}}\mathbb{P}\left(
Y_{n,r}^{\ast}>u\right)  du\\
&  =m_{n,r}-\left(  nc\nu_{n,r}\right)  ^{-1}\exp\left(  -nc\nu_{n,r}\left(
f_{n,r}-m_{n,r}\right)  \right)  +\left(  nc\nu_{n,r}\right)  ^{-1}\exp\left(
-nc\nu_{n,r}f_{n,r}\right) \\
&  +\int_{m_{n,r}}^{M_{n,r}}\mathbb{P}\left(  Y_{n,r}^{\ast}>u\right)  du.
\end{align*}
Hence,%
\begin{align*}
&  \underset{1\leq r\leq k_{n}}{\max}\left|  \mathbb{E}\left(  Z_{n,r}^{\ast
}\right)  -\mu_{n,r}+1\right| \\
&  \leq\underset{1\leq r\leq k_{n}}{\max}\left|  nc\nu_{n,r}\left(
f_{n,r}-m_{n,r}\right)  -1+\exp\left(  -nc\nu_{n,r}\left(  f_{n,r}%
-m_{n,r}\right)  \right)  \right| \\
&  +\exp\left(  -nc\nu_{n}m\right)  +nc\underset{1\leq r\leq k_{n}}{\max}%
\nu_{n,r}\mathbb{P}\left(  Y_{n,r}^{\ast}>m_{n,r}\right)  \left(
M_{n,r}-m_{n,r}\right) \\
&  \leq\underset{0\leq t\leq cn\delta_{n}}{\max}\left(  e^{-t}+t-1\right)
+\exp\left(  -nc\nu_{n}m\right)  +O\left(  \left(  n\delta_{n}\right)
^{2}\right) \\
&  =O\left(  \max\left(  \exp\left(  -mnc\nu_{n}\right)  ;\left(  n\delta
_{n}\right)  ^{2}\right)  \right)  .
\end{align*}
\end{proof}

\noindent In the next lemma we give an uniform upper bound on the centered
moments of $\left(  \xi_{n,r}\right)  $ and an exact uniform control of the
variances and expectations.

\begin{lemma}
\label{moml}Under assumptions $\left(  \mathrm{H.1}\right)  $\ and\ $\left(
\mathrm{H.2}\right)  $ we have\newline i) $\underset{n\rightarrow\infty}%
{\lim\sup\;}\underset{\ell\geq2}{\max}\frac{1}{6^{\ell}\ell!}\;\underset{1\leq
r\leq k_{n}}{\max}\mathbb{E}\left(  \left|  \xi_{n,r}-\mathbb{E}\left(
\xi_{n,r}\right)  \right|  ^{\ell}\right)  <1.\ $\newline ii) $\underset{1\leq
r\leq k_{n}}{\max}\left|  \mathbb{V}\left(  \xi_{n,r}\right)  -1\right|
=o\left(  1\right)  .$\newline iii) $\underset{1\leq r\leq k_{n}}{\max}\left|
\mathbb{E}\left(  \xi_{n,r}-\mu_{n,r}\right)  \right|  =O\left(  \max\left(
\left(  n\delta_{n}\right)  ^{2},n\nu_{n}\exp\left(  -mcn\nu_{n}\right)
,\Delta_{n}\right)  \right)  .$
\end{lemma}

\begin{proof}
i) It is easy to see that
\[
\underset{1\leq r\leq k_{n}}{\max}\mathbb{E}\left(  \left|  \xi_{n,r}%
-\mathbb{E}\left(  \xi_{n,r}\right)  \right|  ^{\ell}\right)  \leq3^{\ell}%
\max\left\{
\begin{array}
[c]{c}%
\underset{1\leq r\leq k_{n}}{\max}\mathbb{E}\left(  \left|  Z_{n,r}%
^{-}-\mathbb{E}\left(  Z_{n,r}^{-}\right)  \right|  ^{\ell}\right)  ,\\
2^{\ell}\underset{1\leq r\leq k_{n}}{\max}\,\mathbb{E}\left(  \left|
\frac{Z_{n,r}^{-}}{N_{n,r}^{-}}\right|  ^{\ell}\right)  ,\;2^{\ell}%
\underset{1\leq r\leq k_{n}}{\max}\,\mathbb{E}\left(  \left|  \xi_{n,r}%
-\xi_{n,r}^{-}\right|  ^{\ell}\right)
\end{array}
\right\}
\]
\noindent and the result follows from Lemma $\ref{mom}$ vii), ix) and x).
\newline ii) Introduce for the sake of simplicity $\gamma_{n,r}=\xi
_{n,r}-Z_{n,r}^{-}=(\xi_{n,r}-\xi_{n,r}^{-})+(Z_{n,r}^{-}/N_{n,r}^{-})$. This
yields
\[
\mathbb{V}\left(  \xi_{n,r}\right)  =\mathbb{V}\left(  Z_{n,r}^{-}\right)
+2\mathbb{C}ov\left(  Z_{n,r}^{-},\gamma_{n,r}\right)  +\mathbb{V}\left(
\gamma_{n,r}\right)
\]
and we thus have,
\[
\underset{1\leq r\leq k_{n}}{\max}\left|  \mathbb{V}\left(  \xi_{n,r}\right)
-1\right|  \leq\underset{1\leq r\leq k_{n}}{\max}\left|  \mathbb{V}\left(
Z_{n,r}^{-}\right)  -1\right|  +\underset{1\leq r\leq k_{n}}{\max}%
\mathbb{V}\left(  \gamma_{n,r}\right)  +2\left[  \underset{1\leq r\leq k_{n}%
}{\max}\mathbb{V}\left(  \gamma_{n,r}\right)  \mathbb{V}\left(  Z_{n,r}%
^{-}\right)  \right]  ^{1/2}.
\]
Lemma $\ref{mom}$ v) shows that
\[
\underset{1\leq r\leq k_{n}}{\max}\left|  \mathbb{V}\left(  Z_{n,r}%
^{-}\right)  -1\right|  =\underset{1\leq r\leq k_{n}}{\max}\left|
2\lambda_{n,r}e^{-\lambda_{n,r}}+e^{-2\lambda_{n,r}}\right|  =o\left(
1\right)  .
\]
Besides,
\[
\underset{1\leq r\leq k_{n}}{\max}\mathbb{V}\left(  \gamma_{n,r}\right)
\leq2\underset{1\leq r\leq k_{n}}{\max}\mathbb{V}\left(  \frac{Z_{n,r}^{-}%
}{N_{n,r}^{-}}\right)  +2\underset{1\leq r\leq k_{n}}{\max}\mathbb{E}\left(
|\xi_{n,r}-\xi_{n,r}^{-}|^{2}\right)  =o(1)
\]
by Lemma \ref{mom} viii) and x) in the particular case where $\ell
=2$.\newline iii) First, by the triangle inequality and Lemma $\ref{mom}$
(vi),%
\begin{align}
\underset{1\leq r\leq k_{n}}{\max}\left|  \mathbb{E}\left(  \frac
{Z_{n,r}^{\ast}}{N_{n,r}}\right)  -1\right|   &  \leq\underset{1\leq r\leq
k_{n}}{\max}\left|  \mathbb{E}\left(  \left(  \frac{Z_{n,r}^{+}}{N_{n,r}%
}-1\right)  \mathbf{1}_{\left\{  N_{n,r}^{+}>0\right\}  }\right)  \right|
\nonumber\\
&+\underset{1\leq r\leq k_{n}}{\max}\left|  \mathbb{E}\left(  \left(
\frac{Z_{n,r}^{-}}{N_{n,r}^{-}}-1\right)  \mathbf{1}_{\left\{  N_{n,r}%
^{+}=0\right\}  }\right)  \right| \nonumber\\
&  \leq\underset{1\leq r\leq k_{n}}{\max}\left|  \mathbb{E}\left(  \left(
\frac{Z_{n,r}^{+}-\mathbb{E}\left(  N_{n,r}\right)  }{N_{n,r}}\right)
\mathbf{1}_{\left\{  N_{n,r}^{+}>0\right\}  }\right)  \right| \nonumber\\
&  +\underset{1\leq r\leq k_{n}}{\max}\left|  \mathbb{E}\left(  \left(
\frac{N_{n,r}-\mathbb{E}\left(  N_{n,r}\right)  }{N_{n,r}}\right)
\mathbf{1}_{\left\{  N_{n,r}^{+}>0\right\}  }\right)  \right|  +O\left(
n\nu_{n}\exp\left(  -mcn\nu_{n}\right)  \right)  . \label{a1}%
\end{align}
Now, since $\mathbb{E}\left(  N_{n,r}\right)  =\mu_{n,r},$ we get using Lemma
\ref{mom} (ii) and $\left(  \ref{dobe6}\right)  $%
\begin{align}
\underset{1\leq r\leq k_{n}}{\max}\left|  \mathbb{E}\left(  \left(
\frac{Z_{n,r}^{+}-\mathbb{E}\left(  N_{n,r}\right)  }{N_{n,r}}\right)
\mathbf{1}_{\left\{  N_{n,r}^{+}>0\right\}  }\right)  \right|   &  \leq
nc\delta_{n}\underset{1\leq r\leq k_{n}}{\max}\mathbb{E}\left(  \left(
N_{n,r}^{-}+1\right)  ^{-1}\right)  \mathbb{P}\left(  N_{n,r}^{+}>0\right)
\nonumber\\
&  =o\left(  \left(  nc\delta_{n}\right)  ^{2}\right)  . \label{a2}%
\end{align}
Moreover,
\begin{align}
\underset{1\leq r\leq k_{n}}{\max}\left|  \mathbb{E}\left(  \left(
\frac{N_{n,r}-\mathbb{E}\left(  N_{n,r}\right)  }{N_{n,r}}\right)
\mathbf{1}_{\left\{  N_{n,r}^{+}>0\right\}  }\right)  \right|   &
\leq\underset{1\leq r\leq k_{n}}{\max}\left|  \mathbb{E}\left(  \left(
\frac{N_{n,r}^{+}-\mathbb{E}\left(  N_{n,r}^{+}\right)  }{N_{n,r}}\right)
\mathbf{1}_{\left\{  N_{n,r}^{+}>0\right\}  }\right)  \right| \nonumber\\
&  +\underset{1\leq r\leq k_{n}}{\max}\left|  \mathbb{E}\left(  \left(
\frac{N_{n,r}^{-}-\mathbb{E}\left(  N_{n,r}^{-}\right)  }{N_{n,r}}\right)
\mathbf{1}_{\left\{  N_{n,r}^{+}>0\right\}  }\right)  \right|  \label{a3}%
\end{align}
and since for all large $n,$%
$
\underset{1\leq r\leq k_{n}}{\max}\mathbb{E}\left(  N_{n,r}^{+}\right)  \leq
nc\delta_{n}<1,
$
we get eventually, using $\left(  \ref{dobe6}\right)  $ again,%
\begin{align}
\underset{1\leq r\leq k_{n}}{\max}\left|  \mathbb{E}\left(  \left(
\frac{N_{n,r}^{+}-\mathbb{E}\left(  N_{n,r}^{+}\right)  }{N_{n,r}}\right)
\mathbf{1}_{\left\{  N_{n,r}^{+}>0\right\}  }\right)  \right|   &
=\underset{1\leq r\leq k_{n}}{\max}\mathbb{E}\left(  \left(  \frac{N_{n,r}%
^{+}-\mathbb{E}\left(  N_{n,r}^{+}\right)  }{N_{n,r}}\right)  \mathbf{1}%
_{\left\{  N_{n,r}^{+}>0\right\}  }\right) \nonumber\\
&  \leq\underset{1\leq r\leq k_{n}}{\max}\mathbb{E}\left(  \left(
\frac{N_{n,r}^{+}-\mathbb{E}\left(  N_{n,r}^{+}\right)  }{N_{n,r}^{-}%
+1}\right)  \mathbf{1}_{\left\{  N_{n,r}^{+}>0\right\}  }\right) \nonumber\\
&  \leq\underset{1\leq r\leq k_{n}}{\max}\mathbb{E}\left(  \left(  N_{n,r}%
^{-}+1\right)  ^{-1}\right)  \mathbb{E}\left(  N_{n,r}^{+}\right) \nonumber\\
&  \leq\underset{1\leq r\leq k_{n}}{\max}\lambda_{n,r}^{-1}nc\nu_{n,r}\left(
M_{n,r}-m_{n,r}\right) \nonumber\\
&  =O\left(  \Delta_{n}\right)  . \label{a4}%
\end{align}
Finally, since, for all $r\leq k_{n},$%
\begin{align}
&  \left|  \mathbb{E}\left(  \left(  \frac{N_{n,r}^{-}-\mathbb{E}\left(
N_{n,r}^{-}\right)  }{N_{n,r}}\right)  \mathbf{1}_{\left\{  N_{n,r}%
^{+}>0\right\}  }\right)  \right| \nonumber\\
&  \leq\left|  \mathbb{E}\left(  \frac{N_{n,r}^{-}-\mathbb{E}\left(
N_{n,r}^{-}\right)  }{N_{n,r}^{-}+1}\right)  \right|  \mathbb{P}\left(
N_{n,r}^{+}=1\right)  +\underset{j\geq2}{\sum}\mathbb{E}\left(  \frac{\left|
N_{n,r}^{-}-\mathbb{E}\left(  N_{n,r}^{-}\right)  \right|  }{N_{n,r}^{-}%
+j}\right)  \mathbb{P}\left(  N_{n,r}^{+}=j\right) \nonumber\\
&  \leq\left|  \mathbb{E}\left(  \frac{N_{n,r}^{-}-\mathbb{E}\left(
N_{n,r}^{-}\right)  }{N_{n,r}^{-}+1}\right)  \right|  \mathbb{P}\left(
N_{n,r}^{+}=1\right)  +\mathbb{E}\left(  \frac{\left|  N_{n,r}^{-}%
-\mathbb{E}\left(  N_{n,r}^{-}\right)  \right|  }{N_{n,r}^{-}+1}\right)
\mathbb{P}\left(  N_{n,r}^{+}\geq2\right) \nonumber\\
&  \leq\left|  \mathbb{E}\left(  \frac{N_{n,r}^{-}-\mathbb{E}\left(
N_{n,r}^{-}\right)  }{N_{n,r}^{-}+1}\right)  \right|  \mathbb{P}\left(
N_{n,r}^{+}=1\right)  +\left(  \mathbb{V}\left(  N_{n,r}^{-}\right)
\mathbb{E}\left(  \left(  N_{n,r}^{-}+1\right)  ^{-2}\right)  \right)
^{1/2}\mathbb{P}\left(  N_{n,r}^{+}\geq2\right) \nonumber\\
&  \leq\left|  \mathbb{E}\left(  \frac{N_{n,r}^{-}-\mathbb{E}\left(
N_{n,r}^{-}\right)  }{N_{n,r}^{-}+1}\right)  \right|  \mathbb{P}\left(
N_{n,r}^{+}=1\right)  +o\left(  \mathbb{P}\left(  N_{n,r}^{+}\geq2\right)
\right)  \label{a5}%
\end{align}
with
\begin{align}
\mathbb{E}\left(  \frac{N_{n,r}^{-}-\mathbb{E}\left(  N_{n,r}^{-}\right)
}{N_{n,r}^{-}+1}\right)   &  =\underset{q=0}{\overset{\infty}{\sum}}\left(
q-\lambda_{n,r}\right)  \frac{\lambda_{n,r}^{q}}{\left(  q+1\right)
!}e^{-\lambda_{n,r}}\nonumber\\
&  =\lambda_{n,r}^{-1}\underset{q=0}{\overset{\infty}{\sum}}\left(
q+1-1\right)  \frac{\lambda_{n,r}^{q+1}}{\left(  q+1\right)  !}e^{-\lambda
_{n,r}}-\underset{q=0}{\overset{\infty}{\sum}}\frac{\lambda_{n,r}^{q+1}%
}{\left(  q+1\right)  !}e^{-\lambda_{n,r}}\nonumber\\
&  =\lambda_{n,r}^{-1}\underset{q=1}{\overset{\infty}{\sum}}\left(
q-1\right)  \frac{\lambda_{n,r}^{q}}{q!}e^{-\lambda_{n,r}}-\underset
{q=1}{\overset{\infty}{\sum}}\frac{\lambda_{n,r}^{q}}{q!}e^{-\lambda_{n,r}%
}\nonumber\\
&  =\lambda_{n,r}^{-1}\left(  \lambda_{n,r}-1+e^{-\lambda_{n,r}}\right)
-\left(  1-e^{-\lambda_{n,r}}\right) \nonumber\\
&  =-\lambda_{n,r}^{-1}+e^{-\lambda_{n,r}}\left(  1+\lambda_{n,r}^{-1}\right)
, \label{a6}%
\end{align}
and
\begin{align}
\mathbb{P}\left(  N_{n,r}^{+}\geq2\right)   &  \leq\underset{q=2}%
{\overset{\infty}{\sum}}\frac{\left(  nc\delta_{n}\right)  ^{q}}%
{q!}e^{-nc\delta_{n}}
 \leq\left(  nc\delta_{n}\right)  ^{2} \label{a7}%
\end{align}
we get collecting $\left(  \ref{a1}\right)  -\left(  \ref{a7}\right)  $ that
\[
\underset{1\leq r\leq k_{n}}{\max}\left|  \mathbb{E}\left(  \frac
{Z_{n,r}^{\ast}}{N_{n,r}}\right)  -1\right|  =O\left(  \max\left(  \left(
n\delta_{n}\right)  ^{2},\Delta_{n}\right)  \right).
\]
As a conclusion, Lemma \ref{mom} (xi) yields$:$%
\begin{align*}
\underset{1\leq r\leq k_{n}}{\max}\left|  \mathbb{E}\left(  \xi_{n,r}%
-\mu_{n,r}\right)  \right|   &  \leq\underset{1\leq r\leq k_{n}}{\max}\left|
\mathbb{E}\left(  Z_{n,r}^{\ast}\right)  -\mu_{n,r}+1\right|  +\underset{1\leq
r\leq k_{n}}{\max}\left|  \mathbb{E}\left(  \frac{Z_{n,r}^{\ast}}{N_{n,r}%
}\right)  -1\right| \\
&  =O\left(  \max\left(  \left(  n\delta_{n}\right)  ^{2},n\nu_{n}\exp\left(
-mcn\nu_{n}\right)  ,\Delta_{n}\right)  \right)  .
\end{align*}
\end{proof}

\begin{proposition}
\label{TCLseul}Under assumptions $\left(  \mathrm{H.1}\right)  -\left(
\mathrm{H.4}\right)  $ and for all $\left(  x_{1},...,x_{p}\right)  \subset
F,$%
\[
\left\{  \frac{nc}{\kappa_{n}\left(  x_{j}\right)  }\left(  \widehat{f}%
_{n}\left(  x_{j}\right)  -\mathbb{E}\left(  \widehat{f}_{n}\left(
x_{j}\right)  \right)  \right)  :1\leq j\leq p\right\}  \underset{\mathcal{D}%
}{\rightarrow}N\left(  0,\Sigma_{\left(  x_{1},...,x_{p}\right)  }\right)  .
\]
\noindent
\end{proposition}

\begin{proof}
The proof is based on Theorem \ref{GenTCL} in the Appendix: For all $1\leq
r\leq k_{n},$ set\\
$\zeta_{n,r}=\xi_{n,r}-\mathbb{E}\left(  \xi_{n,r}\right)  $
and $w_{n,r}=\,^{t}\left(  w_{n,r}\left(  x_{1}\right)  ,...,w_{n,r}\left(
x_{p}\right)  \right)  .$ It is easily\ seen that
\[
\left\{  \frac{nc}{\kappa_{n}\left(  x_{j}\right)  }\left(  \widehat{f}%
_{n}\left(  x_{j}\right)  -\mathbb{E}\left(  \widehat{f}_{n}\left(
x_{j}\right)  \right)  \right)  :1\leq j\leq p\right\}  =\sum_{r=1}^{k_{n}%
}w_{n,r}\zeta_{n,r},
\]
where the $\left(  \zeta_{n,r}\right)  _{1\leq r\leq k_{n}}$ are independent.
Then $\left(  \mathrm{H.3}\right)  $, $\left(  \mathrm{H.4}\right)  $~and
Lemma\ \ref{moml} i), ii) show that the assumptions of Theorem \ref{GenTCL}
are satisfied.
\end{proof}

\begin{proposition}
\label{centr}Under assumptions $\left(  \mathrm{H.1}\right)  ,$ $\left(
\mathrm{H.2}\right)  ,$ $\left(  \mathrm{H.5}\right)  ,$ $\left(
\mathrm{H.6}\right)  $, we have for all $x\in F$,
\[
\frac{nc}{\kappa_{n}\left(  x\right)  }\left(  \mathbb{E}\left(  \widehat
{f}_{n}\left(  x\right)  \right)  -f\left(  x\right)  \right)  \rightarrow
0\;\mathrm{ as }\; n\to\infty.
\]
\end{proposition}

\begin{proof}
For all $x\in F$, we get, by the triangle inequality and assumption $\left(
\mathrm{H.5}\right)  $,%
\begin{align*}
\frac{nc}{\kappa_{n}\left(  x\right)  }\left|  \mathbb{E}\left(  \widehat
{f}_{n}\left(  x\right)  \right)  -f\left(  x\right)  \right|   &  =\left|
\sum_{r=1}^{k_{n}}w_{n,r}\left(  x\right)  \mathbb{E}\left(  \xi_{n,r}\right)
-\frac{nc}{\kappa_{n}\left(  x\right)  }f\left(  x\right)  \right| \\
&  \leq\left|  \sum_{r=1}^{k_{n}}w_{n,r}\left(  x\right)  \left(
\mathbb{E}\left(  \xi_{n,r}\right)  -\mu_{n,r}\right)  \right|  +\left|
\sum_{r=1}^{k_{n}}w_{n,r}\left(  x\right)  \mu_{n,r}-nc\frac{f\left(
x\right)  }{\kappa_{n}(x)}\right| \\
&  \leq\left(  \sum_{r=1}^{k_{n}}|w_{n,r}(x)|\right)  \max_{1\leq r\leq k_{n}%
}|\mathbb{E}(\xi_{n,r})-\mu_{n,r}|+o(1).
\end{align*}
Lemma \ref{moml} iii) and condition $\left(  \mathrm{H.6}\right)  $ give the result.
\end{proof}

\noindent Theorem \ref{TCL} is a straightforward consequence of
Proposition~\ref{TCLseul} and Proposition~\ref{centr}. The following lemma
shows that $\widehat{c}_{n}$ converges to $c$ almost surely. In particular, it
implies that $c$ can be replaced by $\widehat{c}_{n}$ in the above theorem.

\begin{lemma}
\label{remplace1} Under assumptions $\left(  \mathrm{H.1}\right)  $ and
$\left(  \mathrm{H.2}\right)  $, for all $\delta>0$, there exist
$\alpha_{\delta}>0$ and $n_{\delta}>0$ such that $\forall n\geq n_{\delta
},\;\mathbb{P}(|\widehat{c}_{n}-c|\geq\delta)\leq3\exp{(-n\alpha_{\delta})}. $
\end{lemma}

\begin{proof}
We have
\[
|\widehat{c}_{n}-c|\leq\frac{1}{n\widehat{a}_{n}}|N_{n}(S)-nac|+ac\left|
\frac{1}{\widehat{a}_{n}}-\frac{1}{a}\right|  .
\]
Let $\delta>0$ and $\eta_{\delta}=\min(a/2,\;a\delta/(4c))$. Then,
\[
|\widehat{c}_{n}-c|\leq|\widehat{c}_{n}-c|\mathbf{1}_{\left\{  |\widehat
{a}_{n}-a|>\eta_{\delta}\right\}  }+\left(  \frac{2}{na}|N_{n}(S)-nac|+\frac
{2c\eta_{\delta}}{a}\right)  \mathbf{1}_{\left\{  |\widehat{a}_{n}-a|\leq
\eta_{\delta}\right\}  },
\]
and therefore
\begin{align}
\mathbb{P}(|\widehat{c}_{n}-c|\geq\delta)  &  \leq\mathbb{P}\left(  \frac
{1}{na}|N_{n}(S)-nac|\geq\frac{1}{2}\left(  \delta-\frac{2c\eta_{\delta}}%
{a}\right)  \right)  +\mathbb{P}(|\widehat{a}_{n}-a|>\eta_{\delta})\nonumber\\
&  \leq\mathbb{P}\left(  \frac{N_{n}(S)}{na}\notin\left]  c-\frac{\delta}%
{4},\;c+\frac{\delta}{4}\right[  \right)  +\mathbb{P}(|\widehat{a}_{n}%
-a|>\eta_{\delta}). \label{eqmajo}%
\end{align}
Let us consider the first term of (\ref{eqmajo}). Since $N_{n}(S)$ has a
Poisson distribution with mean $nac$, it can be expanded as $N_{n}%
(S)=\sum_{k=1}^{n}\pi_{k}$, where the $\pi_{k}$ are independent Poisson random
variables with mean $ac$. Introducing $\Lambda_{\pi}(s)=\log\mathbb{E}%
(e^{s\pi})=ac(e^{s}-1)$ and denoting
\[
\Lambda_{\pi}^{\ast}(t)=\sup_{s\in\mathbb{R}}(st-\Lambda_{\pi}(s))=\left\{
\begin{array}
[c]{lll}%
t\log{(t/ac)}-t+ac & \mathrm{if} & t\geq0\\
+\infty & \mathrm{if} & t<0,
\end{array}
\right.
\]
Cramer's theorem (see \textsc{Dembo} \& \textsc{Zeitouni}~(1991),
Theorem~2.2.3) yields
\[
\underset{n\rightarrow\infty}{\lim\sup\;}\frac{1}{n}\log\mathbb{P}\left(
\frac{N_{n}(S)}{na}\notin\left]  c-\frac{\delta}{4},\;c+\frac{\delta}%
{4}\right[  \right)  \leq-\inf\left\{  \Lambda_{\pi}^{\ast}(t),\;t\notin
\left]  c-\frac{\delta}{4},\;c+\frac{\delta}{4}\right[  \right\}  <0.
\]
Consequently, there exists $\alpha_{\delta}^{\prime}>0$ such that
\begin{equation}
\forall n\geq1,\mathbb{P}\left(  \frac{N_{n}(S)}{na}\notin\left]
c-\frac{\delta}{4},\;c+\frac{\delta}{4}\right[  \right)  \leq\exp
{(-n\alpha_{\delta}^{\prime})}. \label{eqmajo2}%
\end{equation}
Consider now the second term of (\ref{eqmajo}) and observe that
\[
\mathbb{E}(\widehat{a}_{n})-a=\frac{1}{nc}\sum_{r=1}^{k_{n}}(\mathbb{E}%
(\xi_{n,r})-\mu_{n,r})
\]
Lemma~\ref{moml} iii) implies that
\[
|\mathbb{E}(\widehat{a}_{n})-a|=\frac{k_n}{n}
\left(  \max\left(  \Delta_{n},(n\delta_{n})^2,
n\nu_{n}\exp\left(  -mnc\nu_{n}\right)  \right)  \right)
\]
which converges to 0 under $\left(  \mathrm{H.1}\right)  $ and $\left(
\mathrm{H.2}\right)  $. Therefore, there exists $n_{\delta}>0$ such that
\[
\forall n\geq n_{\delta},\;\mathbb{P}(|\widehat{a}_{n}-a|\geq\eta_{\delta
})\leq\mathbb{P}(|\widehat{a}_{n}-\mathbb{E}(\widehat{a}_{n})|\geq\eta
_{\delta}/2)=\mathbb{P}\left(  \left|  \sum_{r=1}^{k_{n}}((\xi_{n,r}%
-\mathbb{E}(\xi_{n,r}))\right|  \geq\eta_{\delta}nc/2\right)  ,
\]
and in view of Lemma~\ref{moml} i), applying Bernstein's inequality (see
\textsc{Shorack\thinspace\&\thinspace Wellner}~(1986), p.~855) yields that for
some constants $C_{1}$ and $C_{2},$%
\begin{equation}
\forall n\geq n_{\delta},\;\mathbb{P}(|\widehat{a}_{n}-a|\geq\eta_{\delta
})\leq2\exp{\left(  -\frac{\eta_{\delta}^{2}c^{2}n^{2}}{C_{1}k_{n}+C_{2}%
\eta_{\delta}cn}\right)  }\leq2\exp{\left(  -\frac{\eta_{\delta}^{2}c^{2}%
}{C_{1}+C_{2}\eta_{\delta}c}n\right)  }. \label{eqmajo3}%
\end{equation}
Defining $\alpha_{\delta}=\min(\alpha_{\delta}^{\prime},\eta_{\delta}^{2}%
c^{2}/(C_{1}+C_{2}\eta_{\delta}c))$ and collecting (\ref{eqmajo}%
)--(\ref{eqmajo3}) give the result.
\end{proof}

\begin{proof}
[Proof of Corollary \ref{coromain}] It remains to verify that the difference
\[
D_{n}(x)=\frac{n}{\kappa_{n}(x)}(\widehat{c}_{n}-c)(\widehat{f}_{n}\left(
x\right)  -f(x))
\]
can be neglected in the central limit theorem. In this aim, let $(x_{1}%
,\dots,x_{p})\subset F$. For all $\eta>0$ and $\delta>0$, we have
\[
\mathbb{P}\left(  \max_{1\leq j\leq p}|D_{n}(x_{j})|\geq\eta\right)
\leq\mathbb{P}(|\widehat{c}_{n}-c|\geq\delta)+\mathbb{P}\left(  \max_{1\leq
j\leq p}\frac{nc}{\kappa_{n}(x_{j})}|\widehat{f}_{n}(x_{j})-f(x_{j})|\geq\eta
c/\delta\right)  ,
\]
where the first term converges to 0 as $n\rightarrow\infty$ in view of
Lemma~\ref{remplace1}. Thus,
\begin{align}
\underset{n\rightarrow\infty}{\lim\sup\;}\mathbb{P}\left(  \max_{1\leq j\leq
p}|D_{n}(x_{j})|\geq\eta\right)   &  \leq\underset{n\rightarrow\infty}%
{\lim\sup\;}\mathbb{P}\left(  \max_{1\leq j\leq p}\frac{nc}{\kappa_{n}(x_{j}%
)}|\widehat{f}_{n}(x_{j})-f(x_{j})|\geq\eta c/\delta\right)  \nonumber\\
&  =\mathbb{P}\left(  \max_{1\leq j\leq p}|G_{j}|\geq\eta c/\delta\right)
,\label{eqqq}%
\end{align}
where $(G_{1},\dots,G_{p})$ follows the distribution $N(0,\Sigma_{(x_{1}%
,\dots,x_{p})})$ under the conditions of Theorem~\ref{TCL}. Letting
$\delta\rightarrow0$ in (\ref{eqqq}) yields
\[
\underset{n\rightarrow\infty}{\lim\sup\;}\mathbb{P}\left(  \max_{1\leq j\leq
p}|D_{n}(x_{j})|\geq\eta\right)  =0,
\]
and therefore $\{D_{n}(x_{j}),\;1\leq j\leq p\}\overset{\mathbb{P}%
}{\longrightarrow}0.$
\end{proof}

\section{Applications}
\label{applications}

We first introduce a general class of kernel estimators which
will be shown to satisfy our main result given in Theorem~\ref{TCL}.
Then, we focus on the
particular cases of Parzen-Rosenblatt and Dirichlet kernels.  

\subsection{General kernel estimates}

Consider an unbiaised version of Geffroy's estimator: 
\[
\overline{f}_{n}\left(  x\right)  =\sum_{r=1}^{k_{n}}\mathbf{1}_{I_{n,r}%
}\left(  x\right)  \left(  1+N_{n,r}^{-1}\right)  Y_{n,r}^{\ast}.
\]
In order to smooth this estimator, a sequence
$
K_{n}:E\times E\rightarrow\mathbb{R},
$
of general smoothing kernels is introduced.
Conditions on this sequence will be imposed later.
The general kernel estimate is defined by
\begin{align}
\widehat{f}_{n}(x) &  =\int_{E}K_{n}\left(  x,t\right)  \overline{f}%
_{n}\left(  t\right)  \nu\left(  dt\right)  \nonumber\\
&  =\sum_{r=1}^{k_{n}}\left(  \int_{I_{n,r}}K_{n}\left(  x,t\right)
\nu\left(  dt\right)  \right)  \left(  1+N_{n,r}^{-1}\right)  Y_{n,r}^{\ast
}.\label{IntEst}%
\end{align}
It appears that (\ref{IntEst}) is a particular case of (\ref{defest}) with
$\kappa_{n,r}\left(  x\right)  =\nu_{n,r}^{-1}\int_{I_{n,r}}K_{n}\left(
x,t\right)  \nu\left(  dt\right)  .$ In the case where the calculation of this
mean value is computationally expensive, it can be approximated by
$K_{n}\left(  x,x_{n,r}\right)  $ for some $x_{n,r}\in I_{n,r}$, leading to
the simplified estimate
\begin{equation}
\widetilde{f}_{n}(x)=\sum_{r=1}^{k_{n}}\nu_{n,r}\;K_{n}\left(  x,x_{n,r}%
\right)  \;\left(  1+N_{n,r}^{-1}\right)  Y_{n,r}^{\ast},\label{BadEst}%
\end{equation}
which is still a particular case of (\ref{defest}) with $\kappa_{n,r}\left(
x\right)  =K_{n}\left(  x,x_{n,r}\right)  .$\newline In order to introduce the
assumptions needed on $K_{n},$ we set, for all $x\in E,$%
\[
\Gamma_{n}\left(  x\right)  =\;\underset{r\leq k_{n}}{\max}\sup\left\{
K_{n}(x,t)-K_{n}(x,s):\left(  s,t\right)  \in I_{n,r}\times I_{n,r}\right\}
\]
and
\[
\Psi_{n}\left(  x\right)  =\left|  \int_{E}K_{n}(x,t)f\left(  t\right)
\nu\left(  dt\right)  -f\left(  x\right)  \right|  .
\]
For the sake of simplicity,  assume that, for all $n\geq1,$ the
partitions $\left\{  I_{n,r}:\;1\leq r\leq k_{n}\right\}  $ are such that
$\nu_{n,r}=k_{n}^{-1}$ for all $r\leq k_{n}$.  Finally, for all function
$
g:E\rightarrow\mathbb{R},
$
we note
\[
\left\|  g\right\|  _{1}=\int_{E}\left|  g\left(  t\right)  \right|
\nu\left(  dt\right)  ,\;\;\left\|  g\right\|  _{2}=\left(  \int_{E}g\left(
t\right)  ^{2}\nu\left(  dt\right)  \right)  ^{1/2}\;\;\mathrm{and}%
\;\;\left\|  g\right\|  _{E}=\,\underset{t\in E}{\sup}\left|  g\left(
t\right)  \right|  .
\]
In this context, the general assumptions (H.1)--(H.6) can be simplified as:\\
$\left(  \mathrm{H}^{\prime}\mathrm{.1}\right)  \;k_{n}\uparrow\infty$ and
$n^{-1}k_{n}\log\left(  n\right)  \rightarrow0$ as $n\rightarrow\infty
$.\medskip\newline $\left(  \mathrm{H.2}\right)  \;$~$0<m\leq M<+\infty$ and
$nk_{n}^{-1}\Delta_{n}\rightarrow0$ as $n\rightarrow\infty$.\medskip
\newline $\left(  \mathrm{K.0}\right)  $ For all $n\geq1,$ $\int_{E\times
E}\left|  K_{n}(x,t)\right|  \nu\left(  dx\right)  \nu\left(  dt\right)
<\infty.$\newline $\left(  \mathrm{K.1}\right)  \;$For all $\left(
x_{1},x_{2}\right)  \in F\times F$,
\[
\Gamma_{n}\left(  x_{1}\right)  \left\|  K_{n}(x_{2},\;\mathbf{.}\;)\right\|
_{1}=o\left(  \left\|  K_{n}(x_{1},\;\mathbf{.}\;)\right\|  _{2}\left\|
K_{n}(x_{2},\;\mathbf{.}\;)\right\|  _{2}\right)  \;\mathrm{as\;}%
n\rightarrow\infty.
\]
$\left(  \mathrm{K.2}\right)  \;$For all $\left(  x_{1},x_{2}\right)  \in
F\times F$,
\[
\left\langle K_{n}(x_{1},\;\mathbf{.}\;),K_{n}%
(x_{2},\;\mathbf{.}\;)\right\rangle _{2}\left(  \left\|  K_{n}(x_{1}%
,\;\mathbf{.}\;)\right\|  _{2}\left\|  K_{n}(x_{2},\;\mathbf{.}\;)\right\|
_{2}\right)  ^{-1}\rightarrow\;\sigma(x_{1},x_{2})\;\mathrm{as\;}%
n\rightarrow\infty.
\]
$\left(  \mathrm{K.3}\right)  \;$For all $x\in F$,
\[
k_{n}^{-1/2}\left\|  K_{n}(x,\;\mathbf{.}\;)\right\|  _{2}^{-1}\;\left\|
K_{n}(x,\;\mathbf{.}\;)\right\|  _{E}\rightarrow0\;\mathrm{as\;}%
n\rightarrow\infty.
\]
$\left(  \mathrm{K.4}\right)  \;$For all $x\in F$,
\[
nk_{n}^{-1/2}\left\|  K_{n}(x,\;\mathbf{.}\;)\right\|  _{2}^{-1}\;\max\left(
\Psi_{n}\left(  x\right)  ;\;\Delta_{n}\left\|  K_{n}(x,\;\mathbf{.}%
\;)\right\|  _{1}\right)  \rightarrow0\;\mathrm{as\;}n\rightarrow\infty.
\]
$\left(  \mathrm{K.5}\right)  \;$For all $x\in F$,
\[
nk_{n}^{-1/2}\left\|  K_{n}(x,\;\mathbf{.}\;)\right\|  _{2}^{-1}\left(
\sum_{r=1}^{k_{n}}\;\int_{I_{n,r}}\left(  K_{n}\left(  x,t\right)
-K_{n}\left(  x,x_{n,r}\right)  \right)  \nu\left(  dt\right)  \right)
\rightarrow0\;\mathrm{as\;}n\rightarrow\infty.
\]
The results established in Section~\ref{secmain} yield:
\begin{theorem}
\label{genKern}a) Under $\left(  \mathrm{H}^{\prime}\mathrm{.1}\right)
,\;\left(  \mathrm{H.2}\right)  ,\;\left(  \mathrm{K.0}\right)  -\left(
\mathrm{K.4}\right)  ,$ and for all $\left(  x_{1},...,x_{p}\right)  \subset
F,$%
\begin{equation}
\left\{  nck_{n}^{-1/2}\left\|  K_{n}(x_{j},\;\mathbf{.}\;)\right\|  _{2}%
^{-1}\left(  \widehat{f}_{n}\left(  x_{j}\right)  -f\left(  x_{j}\right)
\right)  :1\leq j\leq p\right\}  \underset{\mathcal{D}}{\rightarrow}N\left(
0,\Sigma_{\left(  x_{1},...,x_{p}\right)  }\right)  .\label{Kern1}%
\end{equation}
b) If, moreover, $\left(  \mathrm{K.5}\right)  $ holds, then for all $\left(
x_{1},...,x_{p}\right)  \subset F,$%
\begin{equation}
\left\{  nck_{n}^{-1/2}\left\|  K_{n}(x_{j},\;\mathbf{.}\;)\right\|  _{2}%
^{-1}\left(  \widetilde{f}_{n}\left(  x_{j}\right)  -f\left(  x_{j}\right)
\right)  :1\leq j\leq p\right\}  \underset{\mathcal{D}}{\rightarrow}N\left(
0,\Sigma_{\left(  x_{1},...,x_{p}\right)  }\right)  .\label{Kern2}%
\end{equation}
c) $\left(  \ref{Kern1}\right)  $ and $\left(  \ref{Kern2}\right)  $ also hold
when $c$ is replaced by $\hat{c}_{n}$.  
\end{theorem}

\begin{proof}
a) For all $x\in F$, we just verify $\left(  \mathrm{H.1}\right)  -\left(
\mathrm{H.6}\right)  $ for
$
\kappa_{n,r}\left(  x\right)  :=k_{n}\int_{I_{n,r}}K_{n}\left(  x,t\right)
\nu\left(  dt\right)  .\\
$
$\left(  \mathrm{H.1}\right)  \;$and $\left(  \mathrm{H.2}\right)  $ hold
trivialy. Moreover, by $\left(  \mathrm{K.1}\right)  $,
\begin{align*}
\sum_{r=1}^{k_{n}}\kappa_{n,r}\left(  x_{1}\right)  \kappa_{n,r}\left(
x_{2}\right)   &  =k_{n}^{2}\sum_{r=1}^{k_{n}}\int_{I_{n,r}}\int_{I_{n,r}%
}K_{n}(x_{1},s)K_{n}(x_{2},t)\nu\left(  ds\right)  \nu\left(  dt\right)  \\
&  =k_{n}\left\langle K_{n}(x_{1},\;\mathbf{.}\;),K_{n}(x_{2},\;\mathbf{.}%
\;)\right\rangle _{2}\\
&  +k_{n}^{2}\sum_{r=1}^{k_{n}}\int_{I_{n,r}}\int_{I_{n,r}}K_{n}%
(x_{2},t)\left(  K_{n}(x_{1},s)-K_{n}(x_{1},t)\right)  \nu\left(  dt\right)
\nu\left(  ds\right)  \\
&  =k_{n}\left\langle K_{n}(x_{1},\;\mathbf{.}\;),K_{n}(x_{2},\;\mathbf{.}%
\;)\right\rangle _{2}+k_{n}O\left(  \sum_{r=1}^{k_{n}}\;\Gamma_{n}\left(
x_{1}\right)  \int_{I_{n,r}}\left|  K_{n}(x_{2},t)\right|  \nu\left(
dt\right)  \right)  \\
&  =k_{n}\left[  \left\langle K_{n}(x_{1},\;\mathbf{.}\;),K_{n}(x_{2}%
,\;\mathbf{.}\;)\right\rangle _{2}+o\left(  \left\|  K_{n}(x_{1}%
,\;\mathbf{.}\;)\right\|  _{2}\left\|  K_{n}(x_{2},\;\mathbf{.}\;)\right\|
_{2}\right)  \right]  .
\end{align*}
Hence,
\begin{equation}
\kappa_{n}\left(  x\right)  =k_{n}^{1/2}\left\|  K_{n}(x,\;\mathbf{.}%
\;)\right\|  _{2}\left(  1+o\left(  1\right)  \right)  ,\label{k2part}%
\end{equation}
and $\left(  \mathrm{K.2}\right)  $ leads to
\[
\sum_{r=1}^{k_{n}}w_{n,r}\left(  x_{1}\right)  w_{n,r}\left(  x_{2}\right)
=\sigma(x_{1},x_{2})+o\left(  1\right)  ,
\]
which is $\left(  \mathrm{H.3}\right)  .$ Now, $\left(  \ref{k2part}
\right)$ entails for all large $n,$%
\begin{align*}
\max_{1\leq r\leq k_{n}}|w_{n,r}(x)| &  \leq2k_{n}^{-1/2}\left\|
K_{n}(x,\;\mathbf{.}\;)\right\|  _{2}^{-1}\;k_{n}\max_{1\leq r\leq k_{n}%
}\left|  \int_{I_{n,r}}K_{n}(x,t)\nu\left(  dt\right)  \right|  \\
&  \leq2k_{n}^{-1/2}\left\|  K_{n}(x,\;\mathbf{.}\;)\right\|  _{2}%
^{-1}\;\left\|  K_{n}(x,\;\mathbf{.}\;)\right\|  _{E}\\
&  \rightarrow0\;\mathrm{as\;}n\rightarrow\infty\;\;\;\;\;\;\mathrm{by}%
\;\left(  \mathrm{K.3}\right)  ,
\end{align*}
i.e. $\left(  \mathrm{H.4}\right)  $ holds. In order to show $\left(
\mathrm{H.5}\right)  ,$ note that using $\left(  \ref{k2part}\right)  $ again
in combination with Fubini Theorem (which holds by $\left(  \mathrm{K.0}%
\right)  $) and the triangle inequality yield%
\begin{align*}
& \frac{n}{\kappa_{n}(x)}\left|  \sum_{r=1}^{k_{n}}\nu_{n,r}\kappa
_{n,r}(x)f_{n,r}-f(x)\right|  \\ &  =\frac{n}{\kappa_{n}(x)}\left|  \sum
_{r=1}^{k_{n}}k_{n}\int_{I_{n,r}\times I_{n,r}}K_{n}\left(  x,t\right)
f\left(  s\right)  \nu\left(  dt\right)  \nu\left(  ds\right)  -f(x)\right|
\\
&  \leq\frac{n}{\kappa_{n}(x)}\left|  \sum_{r=1}^{k_{n}}k_{n}\int
_{I_{n,r}\times I_{n,r}}K_{n}\left(  x,t\right)  \left(  f\left(  s\right)
-f(t)\right)  \nu\left(  dt\right)  \nu\left(  ds\right)  \right|  +\frac
{n}{\kappa_{n}(x)}\Psi_{n}\left(  x\right)  \\
&  \leq2nk_{n}^{-1/2}\left\|  K_{n}(x,\;\mathbf{.}\;)\right\|  _{2}%
^{-1}\left(  \Delta_{n}\sum_{r=1}^{k_{n}}\;\left|  \int_{I_{n,r}}K_{n}\left(
x,t\right)  \nu\left(  dt\right)  \right|  +\Psi_{n}\left(  x\right)  \right)
\\
&  \leq2nk_{n}^{-1/2}\left\|  K_{n}(x,\;\mathbf{.}\;)\right\|  _{2}%
^{-1}\left(  \Delta_{n}\left\|  K_{n}(x,\;\mathbf{.}\;)\right\|  _{1}+\Psi
_{n}\left(  x\right)  \right)  \\
&  \rightarrow0\;\mathrm{as\;}n\rightarrow\infty\;\;\;\;\;\;\mathrm{by}%
\;\left(  \mathrm{K.4}\right)  .
\end{align*}
Finally, we show that $\left(  \mathrm{H.6}\right)  $ holds. Since
$\max\left(  \left(  n\delta_{n}\right)  ^{2},\Delta_{n}\right)  =o\left(
nk_{n}^{-1}\Delta_{n}\right)  ,$ it follows that%
\begin{align*}
&  \sum_{r=1}^{k_{n}}|w_{n,r}(x)|\max\left(  \left(  n\delta_{n}\right)
^{2},\Delta_{n}\right)  \\
&  \leq2k_{n}^{-1/2}\left\|  K_{n}(x,\;\mathbf{.}\;)\right\|  _{2}^{-1}%
\sum_{r=1}^{k_{n}}k_{n}\left|  \int_{I_{n,r}}K_{n}(x,t)\nu\left(  dt\right)
\right|  o\left(  nk_{n}^{-1}\Delta_{n}\right)  \\
&  =o\left(  nk_{n}^{-1/2}\left\|  K_{n}(x,\;\mathbf{.}\;)\right\|  _{2}%
^{-1}\left\|  K_{n}(x,\;\mathbf{.}\;)\right\|  _{1}\Delta_{n}\right)
=o\left(  1\right)  \;\;\mathrm{by}\;\left(  \mathrm{K.4}\right)
\end{align*}
and, since, by $\left(  \mathrm{H}^{\prime}\mathrm{.1}\right)  ,$
$n\exp\left(  -mcnk_{n}^{-1}\right)  \rightarrow0,$
\begin{align*}
\sum_{r=1}^{k_{n}}|w_{n,r}(x)|nk_{n}^{-1}\exp\left(  -mcnk_{n}^{-1}\right)
&  \leq n\exp\left(  -mcnk_{n}^{-1}\right)  \max_{1\leq r\leq k_{n}}%
|w_{n,r}(x)|\\
&  =o\left(  \max_{1\leq r\leq k_{n}}|w_{n,r}(x)|\right)  =o\left(  1\right)
\;\;\mathrm{by}\;\left(  \mathrm{H.4}\right)  .
\end{align*}
b) For all $x\in F$, it is easy to see that
\begin{align*}
&  \frac{n}{\kappa_{n}(x)}\left|  \widehat{f}_{n}(x)-\widetilde{f}%
_{n}(x)\right|  \\
&  \leq4Mnk_{n}^{-1/2}\left\|  K_{n}(x,\;\mathbf{.}\;)\right\|  _{2}%
^{-1}\left(  \sum_{r=1}^{k_{n}}\;\int_{I_{n,r}}\left(  K_{n}(x,t)-K_{n}%
(x,x_{n,r})\right)  \nu\left(  dt\right)  \right)  \\
&  =o\left(  1\right)  \;\;\mathrm{by}\;\left(  \mathrm{K.5}\right)  ,
\end{align*}
which, combined with (a), give the intended result by standard
arguments.\newline c) is straightforward by Corollary~\ref{coromain}.
\end{proof}

\noindent Two illustrations of this result are now provided.  
See \textsc{Menneteau}~(2003b) for other applications.

\subsection{Parzen kernel estimates \label{kernel}}

In the following, we take $E=[0,1]^{d}$ $\left(  d\in\mathbb{N}^{\ast}\right)
,$ $\nu$ is the Lebesgue measure on $E$ and $\left\{  I_{n,r}:\;1\leq r\leq
k_{n}\right\}  $ an adjacent equidistant partition of $E$ such that
$I_{n,r}=\prod_{j=1}^{d}J_{n,r,j}$ where the $J_{n,r,j}$ are interval of
$[0,1]$ of length $k_{n}^{-1/d}$, leading to $\nu_{n,r}=k_{n}^{-1}$ for all
$1\leq r\leq k_{n}$. Besides, we denote by $x_{n,r}$ the center of the cell
$I_{n,r}$, $r=1,\dots,k_{n}$. The multivariate Parzen kernel estimate is then
defined by the kernel
\[
K_{n}\left(  x,t\right)  =\frac{1}{h_{n}^{d}}K^{PR}\left(  \frac{x-t}{h_{n}%
}\right)  ,
\]
where $K^{PR}:\mathbb{R}^{d}\rightarrow\mathbb{R}^{+}$ is a $1$-Lipschitzian
Parzen-Rosenblatt kernel with compact support $\mathcal{K}$, and $(h_{n})$ is
a sequence of positive real numbers tending to zero. It tunes the smoothing
introduced by the kernel. For a review on non-parametric regression, see
\textsc{H\"{a}rdle}~(1990). We suppose that $f$ is $\alpha$-Lipschitzian
($0<\alpha\leq1$), in particular,
\begin{equation}
\Delta_{n}=O\left(  k_{n}^{-\alpha/d}\right)  .\label{fcont}%
\end{equation}

\begin{corollary}
\label{coroPR}
Assume that (i) $n^{-1}k_{n}\log\left(  n\right)  \rightarrow0,$ (ii)
$h_{n}^{d}k_{n}\rightarrow\infty$ and$\;$(iii) $nk_{n}^{-1/2}h_{n}%
^{\alpha+d/2}\rightarrow0$ then for all $\left(  x_{1},...,x_{p}\right)
\subset\overset{\circ}{E}=\left(  0,1\right)  ^{d},$%
\begin{equation}
\left\{ v_n c\;(\widehat{f}_{n}(x_{j})-f(x_{j})):1\leq
j\leq p\right\}  \underset{\mathcal{D}}{\rightarrow}N\left(  0,\left\|
K^{PR}\right\|  _{2}^{2}I_{p}\right)  ,\label{ParKern}%
\end{equation}
where $I_{p}$ is the identity matrix of $\mathbb{R}^{p}$ and
$v_n=n h_n^{d/2} k_n^{-1/2}$. \\ 
The choice $h_{n}=n^{-\frac{1}{\alpha+d}}$ and
$k_{n}=n^{\frac{d}{\alpha+d}}u_{n}^{2}$ 
lead to $v_n=n^{\frac{\alpha}{\alpha+d}}u_{n}^{-1},$
where $u_{n}\rightarrow\infty$ arbitrary slowly. 
\end{corollary}

\begin{proof}
$\left(  \mathrm{K.0}\right)  $ holds trivialy and assumption (i) gives
$\left(  \mathrm{H}^{\prime}\mathrm{.1}\right)  .\;$To show $\left(
\mathrm{H.2}\right)  $, note that $\left(  \ref{fcont}\right)$ entails
\begin{equation}
nk_{n}^{-1}\Delta_{n}=O\left(  nk_{n}^{-\left(  1+\alpha/d\right)  }\right)
\label{Lip}%
\end{equation}
and thus,$\;$by\textrm{\ }$\left(  ii\right)  $\ and\ $\left(  iii\right)  ,$%
\[
nk_{n}^{-\left(  1+\alpha/d\right)  }=\left(  nk_{n}^{-1/2}h_{n}^{\alpha
+d/2}\right)  \left(  h_{n}^{d}k_{n}\right)  ^{-\left(  \frac{\alpha}{d}%
+\frac{1}{2}\right)  }=o\left(  1\right)  .
\]
Let us consider now $\left(  \mathrm{K.1}\right)  -\left(  \mathrm{K.4}%
\right)  $. To this aim, set $x\in\overset{\circ}{E}.$ For large enough $n$
(i.e. such that $\mathcal{K}\subset h_{n}^{-1}\left(  x-E\right)  $)$,$%
\begin{equation}
\left\|  K_{n}(x,\;\mathbf{.}\;)\right\|  _{1}=\int_{h_{n}^{-1}\left(
x-E\right)  }K^{PR}(u)du=1.\label{L1Par}%
\end{equation}%
\begin{equation}
\left\|  K_{n}(x,\;\mathbf{.}\;)\right\|  _{2}=h_{n}^{-\frac{d}{2}}\left(
\int_{h_{n}^{-1}\left(  x-E\right)  }\left(  K^{PR}\right)  ^{2}(u)du\right)
^{1/2}=h_{n}^{-\frac{d}{2}}\left\|  K^{PR}\right\|  _{2}.\label{L2Par}%
\end{equation}%
\begin{align}
\Psi_{n}\left(  x\right)   &  =h_{n}^{-d}\left|  \int_{E}K^{PR}\left(
h_{n}^{-1}\left(  x-t\right)  \right)  \left(  f\left(  t\right)  -f\left(
x\right)  \right)  dt\right|  \nonumber\\
&  =\left|  \int_{\mathcal{K}}K^{PR}\left(  u\right)  \left(  f\left(
x-h_{n}u\right)  -f\left(  x\right)  \right)  du\right|  \nonumber\\
&  =O\left(  h_{n}^{\alpha}\right)  .\label{ParEst}%
\end{align}
Moreover, since $K^{PR}$ is $1$-Lipschitzian$,$%
\begin{equation}
\Gamma_{n}\left(  x\right)  =O\left(  k_{n}^{-1/d}h_{n}^{-\left(  d+1\right)
}\right)  .\label{ModPar}%
\end{equation}
To check $\left(  \mathrm{K.1}\right)  $, take $\left(  x_{1},x_{2}\right)
\in F\times F$, then $\left(  \ref{L1Par}\right)  ,\left(  \ref{L2Par}\right)
,\left(  \ref{ModPar}\right)  $ and $\left(  ii\right)  $ entail
\begin{align*}
\Gamma_{n}\left(  x_{1}\right)  \left\|  K_{n}(x_{2},\;\mathbf{.}\;)\right\|
_{1} &  =O\left(  k_{n}^{-1/d}h_{n}^{-\left(  d+1\right)  }\right)  =\left(
k_{n}h_{n}^{d}\right)  ^{-1/d}O\left(  h_{n}^{-d}\right)  \\
&  =o\left(  \left\|  K_{n}(x_{1},\;\mathbf{.}\;)\right\|  _{2}\left\|
K_{n}(x_{2},\;\mathbf{.}\;)\right\|  _{2}\right)  ,
\end{align*}
$\left(  \mathrm{K.2}\right)  $ follows from the fact that for $x_{1}\neq
x_{2},$ we eventually have,%
\[
\left\langle K_{n}(x_{1},\;\mathbf{.}\;),K_{n}(x_{2},\;\mathbf{.}%
\;)\right\rangle _{2}=h_{n}^{-d}\int_{h_{n}^{-1}\left(  x_{1}-E\right)
}K^{PR}\left(  u\right)  K^{PR}\left(  u+h_{n}^{-1}\left(  x_{1}-x_{2}\right)
\right)  du=0.
\]
For $\left(  \mathrm{K.3}\right)$, note that, for all $x\in F$,
\begin{align*}
k_{n}^{-1/2}\left\|  K_{n}(x,\;\mathbf{.}\;)\right\|  _{2}^{-1}\;\left\|
K_{n}(x,\;\mathbf{.}\;)\right\|  _{E} &  =k_{n}^{-1/2}h_{n}^{d/2}\left\|
K^{PR}\right\|  _{2}^{-1}h_{n}^{-d}\left\|  K^{PR}\right\|  _{E}
 =o\left(  1\right) 
\end{align*}
with $\left(  ii\right)$.
Finally, for all $x\in F$, $\left(  \ref{L1Par}\right)  -\left(
\ref{ParEst}\right)  ,$ (ii) and (iii) entail%
\begin{align*}
nk_{n}^{-1/2}\left\|  K_{n}(x,\;\mathbf{.}\;)\right\|  _{2}^{-1}\;\max\left(
\Psi_{n}\left(  x\right)  ;\;\Delta_{n}\left\|  K_{n}(x,\;\mathbf{.}%
\;)\right\|  _{1}\right)   &  =O\left(  nk_{n}^{-1/2}h_{n}^{d/2}\left(
h_{n}^{\alpha}+k_{n}^{-\alpha/d}\right)  \right)  \\
&  =O\left(  nk_{n}^{-1/2}h_{n}^{\frac{d}{2}+\alpha}\right)  =o\left(
1\right)  .
\end{align*}
The end of the proof is straightforward.
\end{proof}

\noindent From the asymptotical point of view, 
$\widehat{f}_{n}$ is better than
$\widetilde{f}_{n}$ and than the estimator based on Parzen kernel proposed, in
the unidimensional case, by \textsc{Girard} \& \textsc{Jacob}~(2001).
When $f$ is $\alpha$-Lipschitzian,
the speed of convergence of $\widehat{f}_{n}$
can be chosen arbitrarily close to the minimax speed
$n^{-\frac{\alpha}{\alpha+d}}$ 
(see \textsc{H\"ardle} \textit{et al}~(1995b)).
Let us also note that the regularity of $\widehat{f}_{n}$ and 
$\widetilde{f}_{n}$ is determined by the choice of the Parzen-Rosenblatt
kernel.  

\subsection{Projection estimates: Dirichlet kernels}

In the sequel $(b_{n})$ is a sequence of integers tending to infinity. Let
$(e_{j})_{j\in\mathbb{N}}$ be an orthonormal basis of $L^{2}\left(
E,\nu\right)  ${.} The expansion of $f$ on this basis truncated to the $b_{n}$
first terms is noted
\[
f_{n}(x)=\sum_{j=0}^{b_{n}}a_{j}e_{j}(x),\qquad x\in E.
\]
Each $a_{j}=\int_{E}e_{j}(t)f(t)\nu\left(  dt\right)  $ is then estimated by
\[
\widehat{a}_{j,k_{n}}=\sum_{r=1}^{k_{n}}\left(  \int_{I_{n,r}}e_{j}%
(t)\nu\left(  dt\right)  \right)  (1+N_{n,r}^{-1})Y_{n,r}^{\ast},\qquad1\leq
j\leq b_{n},
\]
leading to an estimate $\widehat{f}_{n}(x)$ of $f_{n}(x)$ via:
\begin{equation}
\widehat{f}_{n}(x)=\sum_{\ell=0}^{b_{n}}\widehat{a}_{j,k_{n}}e_{\ell}%
(x)=\sum_{r=1}^{k_{n}}\left(  \int_{I_{n,r}}K_{n}^{D}(x,t)\nu\left(
dt\right)  \right)  (1+N_{n,r}^{-1})Y_{n,r}^{\ast}\label{defbasebis}%
\end{equation}
where $K_{n}^{D}$ the Dirichlet's kernel associated to the orthonormal basis
$(e_{j})_{j\in\mathbb{N}}$ defined by
\begin{equation}
K_{n}^{D}\left(  x,t\right)  =\sum_{j=0}^{b_{n}}e_{j}(x)e_{j}(t),\qquad
(x,t)\in E^{2}.\label{Dkern}%
\end{equation}
It appears that (\ref{defbasebis}) is a particular case of (\ref{IntEst}) with
$K_{n}=K_{n}^{D}.$ Of course, the sometimes more easy to handle estimates
\begin{equation}
\widetilde{f}_{n}(x)=\sum_{r=1}^{k_{n}}\nu_{n,r}K_{n}^{D}(x,x_{n,r}%
)(1+N_{n,r}^{-1})Y_{n,r}^{\ast},\label{defbase}%
\end{equation}
can also be defined$.$ Below, we focus on the trigonometric basis on
$E=[0,1]$, $\nu$ is the Lebesgue measure on $E$, $\{I_{n,r}:\;1\leq r\leq
k_{n}\}$ is the equidistant partition of $E$ and then $\nu_{n,r}=1/k_{n}$ for
all $1\leq r\leq k_{n}$. This basis is 
defined for $x\in [0,1]$ by
\[
e_{0}(x)=1,\qquad e_{2k-1}(x)=\sqrt{2}\cos{(2k\pi x)},\qquad e_{2k}%
(x)=\sqrt{2}\sin{(2k\pi x)},\quad k\geq1.
\]
It is easily seen in that case that the Dirichlet kernel is
\begin{align*}
K_{n}^{D}(x,t) &  =\frac{\sin\left(  \left(  1+b_{n}\right)  \pi\left(
x-t\right)  \right)  }{\sin\left(  \pi\left(  x-t\right)  \right)
}\;\;\;\;\mathrm{for}\;x\neq t\\
&  =1+b_{n}\;\;\;\;\mathrm{if}\;x=t.
\end{align*}
In the following, we assume that $f$ is $C^{2}$. In particular,
\begin{equation}
\Delta_{n}=O\left(  k_{n}^{-1}\right)  .\label{fcontTri}%
\end{equation}
Besides, we introduce the boundary conditions $f(0)=f(1)$ and $f'(0)=f'(1)$.  

\begin{corollary}
Assume that (i) $n^{-1}k_{n}\log\left(  n\right)  \rightarrow0,\;$(ii)
$n^{-1}k_{n}^{2}\rightarrow\infty,$ (iii) $\left(  b_{n}\log\left(
b_{n}\right)  \right)  ^{-1}k_{n}\rightarrow\infty$ and$\;$(iv) $nk_{n}%
^{-1/2}b_{n}^{-3/2}\rightarrow0.\;$Then, for all $\left(  x_{1},...,x_{p}%
\right)  \subset\left[  0,1\right]  ,$%
\begin{equation}
\left\{  v_n c\;(\widehat{f}_{n}%
(x_{j})-f(x_{j})):1\leq j\leq p\right\}  \underset{\mathcal{D}}{\rightarrow
}N\left(  0,I_{p}\right)  , \label{TriKern}%
\end{equation}
where $v_n=n(b_{n}k_{n})^{-1/2}$.  
The choice $b_{n}=n^{\frac{1}{2}}$ and
$k_{n}=n^{\frac{1}{2}}\log\left(  n\right)  u_{n}^{2}$ 
leads to $v_n= n^{\frac{1}{2}}
\log\left(  n\right)  ^{-1/2}u_{n}^{-1}$,
where $u_{n} \rightarrow\infty$ arbitrarily slowly.  
\end{corollary}
\begin{proof}
$\left(  \mathrm{K.0}\right)  $ holds trivialy. Assumptions (i) and (ii)
give $\left(  \mathrm{H}^{\prime}\mathrm{.1}\right)  \;$and $\left(
\mathrm{H.2}\right)  $. The following facts are well known (see e.g.
\textsc{Tolstov}~(1962))%
\begin{equation}
\left\|  K_{n}^{D}(x,\;\mathbf{.}\;)\right\|  _{E}=1+b_{n},\;\;\;\left\|
K_{n}^{D}(x,\;\mathbf{.}\;)\right\|  _{2}=\left(  1+b_{n}\right)
^{1/2},\;\;\;\left\|  K_{n}^{D}(x,\;\mathbf{.}\;)\right\|  _{1}=O\left(
\log\left(  b_{n}\right)  \right)  ,\label{TriNorm}%
\end{equation}%
\begin{equation}
\left\langle K_{n}(x_{1},\;\mathbf{.}\;),K_{n}(x_{2},\;\mathbf{.}%
\;)\right\rangle _{2}=K_{n}(x_{1},x_{2})=o\left(  b_{n}\right)  \;\mathrm{for}%
\;x_{1}\neq x_{2}.\label{TriPs}%
\end{equation}
Since $f$ is $C^{2}$, and taking into account of $f(0)=f(1)$ and $f'(0)=f'(1)$,
a double integration by parts yields,
\[
\underset{j\geq0}{\max}\;j^{2}\int_{0}^{1}f\left(  t\right)  e_{j}\left(
t\right)  dt=O\left(  1\right)  .
\]
Hence,
\begin{align}
\Psi_{n}\left(  x\right)     =\left|  \sum_{j>b_{n}}\int_{0}^{1}f\left(
t\right)  e_{j}\left(  t\right)  dt\;e_{j}\left(  x\right)  \right|
=O\left(  \sum_{j>b_{n}}j^{-2}\right)  =O\left(  b_{n}^{-1}\right)
.\label{TriEst}%
\end{align}
Moreover, since
$
\underset{j\geq1}{\max}\;j^{-1}\left\|  e_{j}^{\prime}\right\|  _{E}=O\left(
1\right)  ,
$
the Taylor formula gives%
\begin{align}
\Gamma_{n}\left(  x\right)   &  \leq\sum_{j=0}^{b_{n}}\left|  e_{j}(x)\right|
\sup\left\{  e_{j}(t)-e_{j}(s):\left(  s,t\right)  \in I_{n,r}\times
I_{n,r}\right\}  \nonumber\\
&  =O\left(  k_{n}^{-1}\sum_{j=0}^{b_{n}}j\right)  =O\left(  k_{n}^{-1}%
b_{n}^{2}\right)  .\label{ModTri}%
\end{align}
Finally, $\left(\ref{TriNorm}\right)  -\left(  \ref{ModTri}\right)$
together with (i)-(iv) imply $\left(  \mathrm{K.1}\right)  -\left(  \mathrm{K.4}\right)$, the proof being similar to the one of Corollary~\ref{coroPR}.
\end{proof}

\noindent In this situation, both estimates $\widehat{f}_{n}$ 
and $\widetilde{f}_{n}$ are $C^\infty$. From the asymptotical point of view, 
$\widehat{f}_{n}$ is better than
$\widetilde{f}_{n}$ and than the estimator based on projections proposed
by \textsc{Girard} \& \textsc{Jacob}~(2003b).
Nevertheless, when $f$ is $C^2$,
the above estimates are suboptimal, since the minimax speed of convergence
is $n^{-2/3}$ (see e.g. \textsc{Hall} \textit{et al}~(1998)).
The use of  expansions on wavelet bases should lead to a better speed of
convergence. This is part of our future work.
\newline We refer to \textsc{Girard} \& \textsc{Menneteau}~(2002)
for a brief comparison on simulations of some of the previous
estimates. Let us also emphasize that in such finite sample situations,
the quality of the estimation strongly depends on the choice of the
hyper-parameters. The estimates of type (b) described in introduction and more
generally the estimates (\ref{defest}) require the choice of two
hyper-parameters: the number of extreme values ($k_{n}$) and a smoothing
parameter ($b_{n}$ or $h_{n}$). Similarly, the estimates of type (a) usually
require to select two hyper-parameters: the rate of decrease of the density
towards 0 (noted $\beta$ in the introduction) and the number of continuous
derivatives of $f$ (noted $q$ in the introduction). In our opinion, one of the
main problems in both cases is now to define an adaptive method for choosing
the hyper-parameters.

\section{Appendix}

We provide a general theorem about the central limit property of
a sequence of random $\mathbb{R}^{p}$ valued vectors
\[
\theta_{n}=\sum_{r=1}^{k_{n}}w_{n,r}\zeta_{n,r},\;\;\;\;\;n\geq1,
\]
where $\left(  w_{n,r}\right)  _{1\leq r\leq k_{n}}$ $\subset\mathbb{R}^{p}$
and $\left(  \zeta_{n,r}\right)  _{1\leq r\leq k_{n}}$ are random variables
such that: \medskip\newline $\left(  \mathrm{A.1}\right)  \;\left(
\zeta_{n,r}\right)  _{1\leq r\leq k_{n}}$ are centered and independent random
variables.\medskip\newline $\left(  \mathrm{A.2}\right)  \;\underset{1\leq
r\leq k_{n}}{\max}\left|  \mathbb{E}\left(  \zeta_{n,r}^{2}\right)  -1\right|
\rightarrow0.$\medskip\newline $\left(  \mathrm{A.3}\right)  \;$ There exists
a covariance matrix $\Sigma$ in $\mathbb{R}^{p}$ such that for all $\lambda
\in\mathbb{R}^{p},$
\[
\sum_{r=1}^{k_{n}}\left\langle w_{n,r},\lambda\right\rangle _{\mathbb{R}^{p}%
}^{2}\rightarrow\;^{t}\lambda\Sigma\lambda.
\]
\medskip$\left(  \mathrm{A.4}\right)  \;\underset{1\leq r\leq k_{n}}{\max
}\left\|  w_{n,r}\right\|  _{\mathbb{R}^{p}}=o\left(  1\right)  .$
\medskip\newline $\left(  \mathrm{A.5}\right)  \;\underset{\alpha
\rightarrow\infty}{\lim\sup\;}\underset{n\rightarrow\infty}{\lim\sup
\;}\underset{1\leq r\leq k_{n}}{\max}\mathbb{E}\left(  \zeta_{n,r}%
^{2}\mathbf{1}_{\left\{  \left|  \zeta_{n,r}\right|  >\alpha\right\}
}\right)  =0.$

\begin{theorem}
\label{GenTCL}Under assumptions $\left(  \mathrm{A.1}\right)  -\left(
\mathrm{A.5}\right)  $, $\theta_{n}\underset{\mathcal{D}}{\rightarrow}N\left(
0,\Sigma\right)  . $
\end{theorem}

\begin{proof}
We have to show that, for all $\lambda\in\mathbb{R}^{p},$%
\begin{equation}
\left\langle \theta_{n},\lambda\right\rangle _{\mathbb{R}^{p}}\underset
{\mathcal{D}}{\rightarrow}N\left(  0,^{t}\lambda\Sigma\lambda\right)  .
\label{gentcl1}%
\end{equation}
Now, by Lindeberg Theorem (see e.g : \textsc{Dudley}~(1989) p. 248), it is
easy to see that $\left(  \ref{gentcl1}\right)  $ holds whenever for all
$\varepsilon>0,$%
\begin{equation}
\underset{n\rightarrow\infty}{\lim\sup\;}\sum_{r=1}^{k_{n}}\left\langle
w_{n,r},\lambda\right\rangle _{\mathbb{R}^{p}}^{2}\mathbb{E}\left(
\zeta_{n,r}^{2}\mathbf{1}_{\left\{  \left|  \left\langle w_{n,r}%
,\lambda\right\rangle _{\mathbb{R}^{p}}\zeta_{n,r}\right|  >\varepsilon
\right\}  }\right)  =0. \label{gentcl2}%
\end{equation}
Fix $\lambda\in\mathbb{R}^{p},$ $\varepsilon>0$ and $\alpha>0.$ Using $\left(
\mathrm{A.4}\right)  ,$ we get for all $n$ large enough and all $1\leq r\leq
k_{n}$ that
$
\mathbf{1}_{\left\{  \left|  \left\langle w_{n,r},\lambda\right\rangle
_{\mathbb{R}^{p}}\zeta_{n,r}\right|  >\varepsilon\right\}  }\leq
\mathbf{1}_{\left\{  \left|  \zeta_{n,r}\right|  >\alpha\right\}  }.
$
Hence,
\begin{align*}
&  \underset{n\rightarrow\infty}{\lim\sup\;}\sum_{r=1}^{k_{n}}\left\langle
w_{n,r},\lambda\right\rangle _{\mathbb{R}^{p}}^{2}\mathbb{E}\left(
\zeta_{n,r}^{2}\mathbf{1}_{\left\{  \left|  \left\langle w_{n,r}%
,\lambda\right\rangle _{\mathbb{R}^{p}}\zeta_{n,r}\right|  >\varepsilon
\right\}  }\right) \\
&  \leq\underset{n\rightarrow\infty}{\lim\sup\;}\left(  \sum_{r=1}^{k_{n}%
}\left\langle w_{n,r},\lambda\right\rangle _{\mathbb{R}^{p}}^{2}\right)
\underset{1\leq r\leq k_{n}}{\max}\mathbb{E}\left(  \zeta_{n,r}^{2}%
\mathbf{1}_{\left\{  \left|  \zeta_{n,r}\right|  >\alpha\right\}  }\right) \\
&  \leq^{t}\lambda\Sigma\lambda\;\underset{n\rightarrow\infty}{\lim\sup
\;}\underset{1\leq r\leq k_{n}}{\max}\mathbb{E}\left(  \zeta_{n,r}%
^{2}\mathbf{1}_{\left\{  \left|  \zeta_{n,r}\right|  >\alpha\right\}
}\right)  .
\end{align*}
and we get the result by $\left(  \mathrm{A.5}\right)  $ when $\alpha
\uparrow\infty.$
\end{proof}

\section*{References}

\begin{description}
\item { Baufays, P. and Rasson, J.P.} (1985) { A new geometric discriminant
rule.} \emph{Computational Statistics Quarterly}, \textbf{2}, 15--30.

\item  Dembo, A. and Zeitouni O. (1991) \emph{Large deviations techniques and
applications},\textit{\ }Jones and Bartlett\textit{.}

\item { Deprins, D., Simar, L. and Tulkens, H.} (1984) { Measuring Labor
Efficiency in Post Offices.} in \emph{The Performance of Public Enterprises:
Concepts and Measurements} by M. Marchand, P. Pestieau and H. Tulkens, North
Holland ed, Amsterdam.

\item { Devroye, L.P. and Wise, G.L.} (1980) { Detection of abnormal behavior
via non parametric estimation of the support.} \emph{SIAM J. Applied Math.},
\textbf{38}, 448--480.

\item  {Dudley, R.M.} (1989) \emph{Real Analysis and Probability}, Wadsworth
and Brooks / Cole.

\item  {Gardes, L.} (2002) { Estimating the support of a Poisson process via
the {F}aber-{S}hauder basis and extreme values.} \emph{Publications de
l'Institut de Statistique de l'Universit\'{e} de Paris}, \textbf{XXXXVI}, 43--72.

\item { Geffroy, J.} (1964) { Sur un probl\`{e}me d'estimation g\'{e}%
om\'{e}trique.} \emph{Publications de l'Institut de Statistique de
l'Universit\'{e} de Paris}, \textbf{XIII}, 191--200.

\item  {Gijbels, I. and Peng, L.} (1999) { Estimation of a support curve via
order statistics.} \emph{Discussion Paper} \textbf{9905}, Institut de
Statistique, Universit\'{e} Catholique de Louvain.

\item { Gijbels, I., Mammen, E., Park, B.U. and Simar, L.} (1999) {On
estimation of monotone and concave frontier functions.} \emph{Journal of the
American Statistical Association}, \textbf{94}, 220--228.

\item { Girard, S. and Jacob, P.} (2003a) { Extreme values and Haar series
estimates of point processes boundaries.} \emph{Scandinavian Journal of
Statistics}, \textbf {30}(2), 369--384.

\item { Girard, S. and Jacob, P.} (2003b) { Projection estimates of point
processes boundaries.} \emph{Journal of Statistical Planning and Inference},
\textbf{116}(1), 1--15.

\item { Girard, S. and Jacob, P.} (2001) { Extreme values and kernel estimates
of point processes boundaries.} \emph{Technical report ENSAM-INRA-UM2},
\textbf{01-02}.

\item { Girard, S. and Menneteau, L.} (2002) { Limit theorems for extreme
values estimates of point processes boundaries.} \emph{Technical report
INRIA}, \textbf{RR-4366}.

\item  {Hall, P., Nussbaum, M. and Stern, S. E.} (1997) {On the estimation of
a support curve of indeterminate sharpness.} \emph{J. Multivariate Anal.},
\textbf{62}, 204--232.

\item { Hall, P., Park, B. U. and Stern, S. E.} (1998) { On polynomial
estimators of frontiers and boundaries.} \emph{J. Multivariate Anal.},
\textbf{66}, 71--98.

\item  {H\"{a}rdle, W.} (1990) \emph{Applied nonparametric regression},
Cambridge University Press, Cambridge.

\item { H\"{a}rdle, W., Hall, P. and Simar, L.} (1995a) { Iterated bootstrap
with application to frontier models.} \emph{J. Productivity Anal.},
\textbf{6}, 63--76.

\item { H\"{a}rdle, W., Park, B. U. and Tsybakov, A. B.} (1995b) { Estimation
of a non sharp support boundaries.} \emph{J. Multivariate Anal.}, \textbf{43}, 205--218.

\item { Hardy, A. and Rasson, J. P.} (1982) { Une nouvelle approche des
probl\`{e}mes de classification automatique.} \emph{Statistique et Analyse des
donn\'{e}es}, \textbf{7}, 41--56.

\item { Hartigan, J. A.} (1975) \emph{Clustering Algorithms.} Wiley, Chichester.

\item { Jacob, P. and Suquet, C.} (1995) { Estimating the edge of a Poisson
process by orthogonal series.} \emph{Journal of Statistical Planning and
Inference}, \textbf{46}, 215--234.

\item { Korostelev, A. P., Simar, L. and Tsybakov, A. B.} (1995) {Efficient
estimation of monotone boundaries.} \emph{The Annals of Statistics},
\textbf{23}, 476--489.

\item  Korostelev, A. P. and Tsybakov, A.B. (1993) { Minimax theory of image
reconstruction.} in \emph{Lecture Notes in Statistics}, \textbf{82},
Springer-Verlag, New York.

\item { Mammen, E. and Tsybakov, A. B.} (1995) {Asymptotical minimax recovery
of sets with smooth boundaries. } \emph{The Annals of Statistics},
\textbf{23}, 502--524.

\item  Menneteau, L. (2003a) Multidimentional limit theorems for smoothed
extreme value estimates of empirical point processes boundaries, submitted.

\item  Menneteau, L. (2003b) Some limit theorems for piecewise constant kernel
smoothed estimates of Poisson point process boundaries, submitted.

\item  Shorack, G. R. and Wellner, J.A. (1986) \ \emph{Empirical processes
with applications to statistics}, Wiley, New York.

\item  Tolstov, G.P. (1962) \ \emph{Fourier Series}, Dover Publications, New York.
\end{description}

\end{document}